\documentclass[reqno,12pt]{amsart}
\usepackage{amscd,amsfonts,amssymb}
\usepackage[T2A]{fontenc}
\usepackage[cp1251]{inputenc}
\numberwithin{equation}{section}
\newtheorem{Theorem}{Theorem}[section]
\newtheorem{Lemma}{Lemma}[section]
\newtheorem{Corollary}{Corollary}[section]

\theoremstyle{definition}

\theoremstyle{remark}

\newtheorem{Example}{Example}[section]

\newcommand{\mes}{\mathop{\rm mes}\nolimits}

\author{A.A. Kon'kov}
\address{Department of Differential Equations,
Faculty of Mechanics and Mathematics,
Mo\-s\-cow Lo\-mo\-no\-sov State University,
Vorobyovy Gory,
Moscow, 119992 Russia.
}
\email{konkov@mech.math.msu.su}
\author{A.E. Shishkov}
\address{
Center of Nonlinear Problems of Mathematical Physics,
RUDN University,
Mi\-klu\-k\-ho-Mak\-la\-ya str. 6,
Moscow, 117198 Russia.
}
\email{aeshkv@yahoo.com}
\thanks{
}
\title[On blow-up conditions]{On blow-up conditions for solutions of systems of quasilinear second-order elliptic inequalities}
\date{}

\begin{document}

\begin{abstract}
We study systems of the differential inequalities
$$
	\left\{
		\begin{aligned}
			&
			- \operatorname{div} A_1 (x, \nabla u_1)
			\ge
			F_1 (x, u_2)
			&
			\mbox{in } {\mathbb R}^n,
			\\
			&
			- \operatorname{div} A_2 (x, \nabla u_2)
			\ge
			F_2 (x, u_1)
			&
			\mbox{in } {\mathbb R}^n,
		\end{aligned}
	\right.
$$
where $n \ge 2$ and $A_i$ are Caratheodory functions such that
$$
	C_1
	|\xi|^{p_i}
	\le
	\xi
	A_i (x, \xi),
	\quad
	|A_i (x, \xi)|
	\le
	C_2
	|\xi|^{p_i - 1},
	\quad
	i = 1,2,
$$ 
with some constants
$C_1, C_2 > 0$
and
$p_1, p_2 > 1$
for almost all
$x \in {\mathbb R}^n$
and for all
$\xi \in {\mathbb R}^n$,
$n \ge 2$. 
For non-negative solutions of these systems we obtain exact blow-up conditions.
\end{abstract}

\maketitle

\section{Introduction}

We consider systems of the differential inequalities
\begin{equation}
	\left\{
		\begin{aligned}
			&
			- \operatorname{div} A_1 (x, \nabla u_1)
			\ge
			F_1 (x, u_2)
			&
			\mbox{in } {\mathbb R}^n,
			\\
			&
			- \operatorname{div} A_2 (x, \nabla u_2)
			\ge
			F_2 (x, u_1)
			&
			\mbox{in } {\mathbb R}^n,
		\end{aligned}
	\right.
	\label{1.1}
\end{equation}
where $n \ge 2$ and $A_k$ are Caratheodory functions such that
\begin{equation}
	C_1
	|\xi|^{p_k}
	\le
	\xi
	A_k (x, \xi),
	\quad
	|A_k (x, \xi)|
	\le
	C_2
	|\xi|^{p_k - 1},
	\quad
	k = 1,2,
	\label{ec}
\end{equation}
with some constants
$C_1, C_2 > 0$
and
$p_1, p_2 > 1$
for almost all
$x \in {\mathbb R}^n$
and for all
$\xi \in {\mathbb R}^n$.
In so doing, the functions $F_1$ and $F_2$ are assumed to be non-negative on ${\mathbb R}^n \times [0, \infty)$, positive on ${\mathbb R}^n \times (0, \infty)$, and non-decreasing with respect to the last argument on the interval $[0, \varepsilon]$ for some real number $0 < \varepsilon < 1$.

An ordered pair 
$
	(u_1, u_2) 
	\in 
	W_{p_1, loc}^1 ({\mathbb R}^n) 
	\times 
	W_{p_2, loc}^1 ({\mathbb R}^n)
$ 
is called a solution of~\eqref{1.1} if 
$
	F_1 (x, u_2), F_2 (x, u_1) 
	\in 
	L_{1, loc} ({\mathbb R}^n)
$ 
and, moreover,
$$
	\int_{
		{\mathbb R}^n
	}
	A_1 (x, \nabla u_1) \nabla \varphi
	\,
	dx
	\ge
	\int_{
		{\mathbb R}^n
	}
	F_1 (x, u_2)
	\varphi
	\,
	dx
$$
and
$$
	\int_{
		{\mathbb R}^n
	}
	A_2 (x, \nabla u_2) \nabla \varphi
	\,
	dx
	\ge
	\int_{
		{\mathbb R}^n
	}
	F_2 (x, u_1)
	\varphi
	\,
	dx
$$
for any non-negative function $\varphi \in C_0^\infty ({\mathbb R}^n)$.

Without loss of generality, we can assume that solutions of~\eqref{1.1} satisfy the relations
\begin{equation}
	\operatorname*{ess\,inf}\limits_{
		{\mathbb R}^n
	}
	u_k
	=
	0,
	\quad
	k = 1, 2.
	\label{1.2}
\end{equation}
If this is not the case, we replace $u_k$ by $u_k - \alpha_k$, where
\begin{equation}
	\alpha_k
	=
	\operatorname*{ess\,inf}\limits_{
		{\mathbb R}^n
	}
	u_k,
	\quad
	k = 1, 2.
	\label{1.3}
\end{equation}
After this replacement,
the left-hand sides of~\eqref{1.1} obviously do not change and the right-hand sides transform to $F_1 (x, u_2 + \alpha_2)$ and $F_2 (x, u_1 + \alpha_1)$.

The absence of solutions of differential equations and inequalities, which is known as the blow-up phenomenon, has been studied by many authors~[1--14].
In most cases, these studies were limited to power-law nonlinearity.
In our paper, we consider the case of the general nonlinearity.
In so doing, we manage to strengthen results of~\cite{AM, ASComPDE}.

Note that the only relevant case is $n > \operatorname{max} \{ p_1, p_2 \}$. Indeed, if $n \le p_k$ for some $k \in \{ 1, 2 \}$, then any non-negative solution of the inequality
$$
	- \operatorname{div} A_k (x, \nabla u) \ge 0
	\quad
	\mbox{in } {\mathbb R}^n
$$
is a constant~\cite{MPbook}.

Below it is assumed that $\theta > 1$ is some given real number and
\begin{equation}
	f_k (r, \zeta)
	=
	\operatorname*{ess\,inf}\limits_{
		x \in B_{\theta r} \setminus B_{r / \theta}
	}
	F_k (x, \zeta),
	\quad
	r, \zeta > 0,
	\;
	k = 1, 2.
	\label{1.4}
\end{equation}
Let us denote by $B_r$ the open ball of radius $r > 0$ centered at zero.
We also assume that for any real numbers $\varepsilon_*, \zeta_* > 0$ there exist a real number $r_* > 0$ and functions
$q_k : [r_*, \infty) \to [0, \infty)$
and
$g_k : [\zeta_*, \infty) \to (0, \infty)$, $k = 1,2$,
satisfying the following conditions:
\begin{enumerate}
\item[$(a)$]
for all $r \in [r_*, \infty)$ and $\zeta \in [\zeta_*, \infty)$ such that
$$
	\left(
		\frac{
			\zeta
		}{
			r^{n - p_2}
		}
	\right)^{1 / (p_2 - 1)}
	\le
	\varepsilon
	\quad
	\mbox{and}
	\quad
	\varepsilon_*
	r^{p_1 / (p_1 - 1)}
	f_1^{1 / (p_1 - 1)}
	\left(
		r, 
		\left(
			\frac{
				\zeta
			}{
				r^{n - p_2}
			}
		\right)^{1 / (p_2 - 1)}
	\right)
	\le
	\varepsilon
$$
we have
\begin{equation}
	f_2 
	\left( 
		r,
		\varepsilon_*
		r^{p_1 / (p_1 - 1)}
		f_1^{1 / (p_1 - 1)}
		\left(
			r,
			\left(
				\frac{
					\zeta
				}{
					r^{n - p_2}
				}
			\right)^{1 / (p_2 - 1)}
		\right)
	\right)
	\ge
	q_1 (r) 
	g_1 (\zeta);
	\label{1.5}
\end{equation}
\item[$(b)$]
for all $r \in [r_*, \infty)$ and $\zeta \in [\zeta_*, \infty)$ such that
$$
	\left(
		\frac{
			\zeta
		}{
			r^{n - p_1}
		}
	\right)^{1 / (p_1 - 1)}
	\le
	\varepsilon
	\quad
	\mbox{and}
	\quad
	\varepsilon_*
	r^{p_2 / (p_2 - 1)}
	f_2^{1 / (p_2 - 1)}
	\left(
		r,
		\left(
			\frac{
				\zeta
			}{
				r^{n - p_1}
			}
		\right)^{1 / (p_1 - 1)}
	\right)
	\le
	\varepsilon
$$
we have
$$
	f_1
	\left(
		r,
		\varepsilon_*
		r^{p_2 / (p_2 - 1)}
		f_2^{1 / (p_2 - 1)}
		\left(
			r,
			\left(
				\frac{
					\zeta
				}{
					r^{n - p_1}
				}
			\right)^{1 / (p_1 - 1)}
		\right)
	\right)
	\ge
	q_2 (r) 
	g_2 (\zeta).
$$
\end{enumerate}

\section{Main results}

\begin{Theorem}\label{T2.1}
Suppose that $n > \operatorname{max} \{ p_1, p_2 \}$. Also let for any real numbers $\varepsilon_*, \zeta_* > 0$ there be a real number $r_* > 0$, non-decreasing functions 
$g_1, g_2 : [\zeta_*, \infty) \to (0, \infty)$, and locally bounded measurable functions 
$q_1, q_2 : [r_*, \infty) \to [0, \infty)$ such that $(a)$ and $(b)$ are valid and, moreover,
\begin{equation}
	\int_{\zeta_*}^\infty
	\frac{
		d\zeta
	}{
		g_k (\zeta)
	}
	<
	\infty
	\label{T2.1.1}
\end{equation}
and
\begin{equation}
	\int_{r_*}^\infty
	r^{n - 1}
	q_k (r)
	\,
	dr
	=
	\infty
	\label{T2.1.2}
\end{equation}
for some $k \in \{ 1, 2 \}$. Then any non-negative solution of~\eqref{1.1}, \eqref{1.2} is identically zero.
\end{Theorem}

The proof of Theorem~\ref{T2.1} is given in Section~\ref{proof}.
A partial case of~\eqref{1.1} is the systems
\begin{equation}
	\left\{
		\begin{aligned}
			&
			- \operatorname{div} A_1 (x, \nabla u_1)
			\ge
			a_1 (x) u_2^{\lambda_1}
			&
			\mbox{in } {\mathbb R}^n,
			\\
			&
			- \operatorname{div} A_2 (x, \nabla u_2)
			\ge
			a_2 (x) u_1^{\lambda_2}
			&
			\mbox{in } {\mathbb R}^n,
		\end{aligned}
	\right.
	\label{2.1}
\end{equation}
where $a_1$ and $a_2$ are non-negative measurable functions and $\lambda_1, \lambda_2 \ge 0$ are real numbers.

\begin{Corollary}\label{C2.1}
Suppose that $n > \operatorname{max} \{ p_1, p_2 \}$ and
\begin{equation}
	\frac{
		\lambda_1 \lambda_2 
	}{
		(p_1 - 1) (p_2 - 1)
	}
	>
	1.
	\label{C2.1.1}
\end{equation}
Also let there be $k \in \{ 1, 2 \}$ such that
\begin{equation}
	\int_1^\infty
	r^{\mu_k}
	\alpha_k (r)
	\,
	dr
	=
	\infty
	\label{C2.1.2},
\end{equation}
where
$$
	\mu_1
	=
	n - 1 
	+
	\frac{
		\lambda_2 p_1
	}{
		p_1 - 1
	}
	-
	\frac{
		\lambda_1 \lambda_2 (n - p_2)
	}{
		(p_1 - 1) (p_2 - 1)
	},
$$
$$
	\mu_2
	=
	n - 1 
	+
	\frac{
		\lambda_1 p_2
	}{
		p_2 - 1
	}
	-
	\frac{
		\lambda_2 \lambda_1 (n - p_1)
	}{
		(p_2 - 1) (p_1 - 1)
	},
$$
\begin{equation}
	\alpha_1 (r)
	=
	\left(
		\operatorname*{ess\,inf}\limits_{
			B_{\theta r} \setminus B_{r / \theta}
		}
		a_1
	\right)^{\lambda_2 / (p_1 - 1)}
	\operatorname*{ess\,inf}\limits_{
		B_{\theta r} \setminus B_{r / \theta}
	}
	a_2,
	\label{C2.1.5},
\end{equation}
and
\begin{equation}
	\alpha_2 (r)
	=
	\left(
		\operatorname*{ess\,inf}\limits_{
			B_{\theta r} \setminus B_{r / \theta}
		}
		a_2
	\right)^{\lambda_1 / (p_2 - 1)}
	\operatorname*{ess\,inf}\limits_{
		B_{\theta r} \setminus B_{r / \theta}
	}
	a_1
	\label{C2.1.6}
\end{equation}
for some real number $\theta > 1$.
Then any non-negative solution of~\eqref{2.1}, \eqref{1.2} is identically zero.
\end{Corollary}

\begin{proof}
We assume, without loss of generality, that~\eqref{C2.1.2} holds for $k = 1$.
If $k = 2$, then all reasoning is absolutely similar.
For systems~\eqref{2.1}, inequality~\eqref{1.5} takes the form
$$
	\varepsilon_*^{\lambda_2}
	\alpha_1 (r)
	r^{\lambda_2 p_1 / (p_1 - 1)}
	\left(
		\frac{
			\zeta
		}{
			r^{n - p_2}
		}
	\right)^{\lambda_1 \lambda_2 / ((p_1 - 1) (p_2 - 1))}
	\ge
	q_1 (r)
	g_1 (\zeta).
$$
Thus, using Theorem~\ref{T2.1} with
$$
	g_1 (\zeta) 
	=
	\zeta^{
		\lambda_1 \lambda_2 / ((p_1 - 1) (p_2 - 1))
	}
	\quad
	\mbox{and}
	\quad
	q_1 (r)
	=
	\varepsilon_*^{\lambda_2}
	r^{
		\lambda_2 p_1 / (p_1 - 1)
		-
		\lambda_1 \lambda_2 (n - p_2) / ((p_1 - 1) (p_2 - 1))
	}
	\alpha_1 (r),
$$
we complete the proof.
\end{proof}

\begin{Corollary}\label{C2.2}
Suppose that $n > \operatorname{max} \{ p_1, p_2 \}$ and
\begin{equation}
	\frac{
		\lambda_1 \lambda_2 
	}{
		(p_1 - 1) (p_2 - 1)
	}
	\le
	1.
	\label{C2.2.1}
\end{equation}
If at least one of the following two relations is valid:
\begin{equation}
	\operatorname*{lim\,sup}\limits_{
		r \to \infty
	}
	r^{
		n
		+
		\lambda_2 p_1 / (p_1 - 1)
		-
		\lambda (n - p_2)
	}
	\alpha_1 (r) 
	>
	0,
	\label{C2.2.2}
\end{equation}
\begin{equation}
	\operatorname*{lim\,sup}\limits_{
		r \to \infty
	}
	r^{
		n
		+
		\lambda_1 p_2 / (p_2 - 1)
		-
		\lambda (n - p_1)
	}
	\alpha_2 (r) 
	>
	0,
	\label{C2.2.3}
\end{equation}
where $\lambda > 1$ is a real number and the functions $\alpha_1$ and $\alpha_2$ are defined by~\eqref{C2.1.5} and~\eqref{C2.1.6} for some real number $\theta > 1$, then any non-negative solution of~\eqref{2.1}, \eqref{1.2} is identically zero.
\end{Corollary}

\begin{proof}
We assume that~\eqref{C2.2.2} is fulfilled. In the case where~\eqref{C2.2.3} holds, our reasoning is absolutely similar. Let $r_i$, $i = 1,2,\ldots$, be a sequence of real numbers such that $r_i \to \infty$ as $i \to \infty$ and, moreover,
$$
	r_i^{
		n
		+
		\lambda_2 p_1 / (p_1 - 1)
		-
		\lambda (n - p_2)
	}
	\alpha_1 (r_i) 
	>
	\gamma
$$
with some constant $\gamma > 0$ for all $i = 1,2, \ldots$.
Putting
$$
	\tilde \alpha_1 (r)
	=
	\left(
		\operatorname*{ess\,inf}\limits_{
			B_{\tilde \theta r} \setminus B_{r / \tilde \theta}
		}
		a_1
	\right)^{\lambda_2 / (p_1 - 1)}
	\operatorname*{ess\,inf}\limits_{
		B_{\tilde \theta r} \setminus B_{r / \tilde \theta}
	}
	a_2,
$$
where $1 < \tilde \theta < \theta$ is some real number, we obviously obtain
$$
	r_i^{
		n
		+
		\lambda_2 p_1 / (p_1 - 1)
		-
		\lambda (n - p_2)
	}
	\inf_{
		(r_i \theta / \tilde \theta, r_i \tilde \theta / \theta)
	}
	\tilde \alpha_1
	>
	\gamma
$$
for all $i = 1,2, \ldots$. Hence,
$$
	\int_{r_0}^\infty
	r^{n-1}
	q_1 (r)
	\,
	dr
	=
	\infty,
$$
where
$$
	q_1 (r)
	=
	r^{
		\lambda_2 p_1 / (p_1 - 1)
		-
		\lambda (n - p_2)
	}
	\tilde \alpha_1 (r).
$$
Since,
$$
	\left(
		\frac{
			\zeta
		}{
			r^{n - p_2}
		}
	\right)^{\lambda_1 \lambda_2 / ((p_1 - 1) (p_2 - 1))}
	\ge
	\left(
		\frac{
			\zeta
		}{
			r^{n - p_2}
		}
	\right)^\lambda
$$
for all real numbers $r > 0$ and $\zeta > 0$ satisfying the condition
\begin{equation}
	\frac{
		\zeta
	}{
		r^{n - p_2}
	}
	\le 
	1,
	\label{PC2.2.1}
\end{equation}
we also have
\begin{equation}
	\tilde \alpha_1 (r)
	r^{\lambda_2 p_1 / (p_1 - 1)}
	\left(
		\frac{
			\zeta
		}{
			r^{n - p_2}
		}
	\right)^{\lambda_1 \lambda_2 / ((p_1 - 1) (p_2 - 1))}
	\ge
	q_1 (r)
	\zeta^\lambda
	\label{PC2.2.2}
\end{equation}
for all real numbers $r > 0$ and $\zeta > 0$ satisfying~\eqref{PC2.2.1}.
It can bee seen that for systems of the form~\eqref{2.1} inequality~\eqref{PC2.2.2} is equivalent to~\eqref{1.5} with
$
	g_1 (\zeta) 
	= 
	\varepsilon_*^{\lambda_2}
	\zeta^\lambda
$
and $\theta$ replaced by $\tilde \theta$.
Thus, to complete the proof, it remains to apply Theorem~\ref{T2.1}.
\end{proof}

\begin{Example}\label{E2.1}
Consider the system
\begin{equation}
	\left\{
		\begin{aligned}
			&
			- \operatorname{div} A_1 (x, \nabla u_1)
			\ge
			u_2^{\lambda_1}
			&
			\mbox{in } {\mathbb R}^n,
			\\
			&
			- \operatorname{div} A_2 (x, \nabla u_2)
			\ge
			u_1^{\lambda_2}
			&
			\mbox{in } {\mathbb R}^n,
		\end{aligned}
	\right.
	\label{E2.1.1}
\end{equation}
where $n > \operatorname{max} \{ p_1, p_2 \}$ and $\lambda_1, \lambda_2 \ge 0$ are real numbers.
By Corollaries~\ref{C2.1} and~\ref{C2.2}, if
\begin{equation}
	\frac{
		n (p_1 - 1) (p_2 - 1)
	}{
		\lambda_1 \lambda_2
	}
	\ge
	\operatorname{min}
	\left\{
		n - p_1 - (p_1 - 1) \frac{p_2}{\lambda_2},
		n - p_2 - (p_2 - 1) \frac{p_1}{\lambda_1}
	\right\},
	\label{E2.1.2}
\end{equation}
then any non-negative solution of~\eqref{E2.1.1}, \eqref{1.2} is identically zero.
Indeed, in the case where~\eqref{C2.1.1} is valid, this follows immediately from Corollary~\ref{C2.1}.
Let~\eqref{C2.2.1} be fulfilled. We obviously have
$$
	n
	+
	\lambda_2 p_1 / (p_1 - 1)
	>
	n - p_2;
$$
therefore, there exists a real numbers $\lambda > 1$ such that
$$
	n
	+
	\lambda_2 p_1 / (p_1 - 1)
	-
	\lambda (n - p_2)
	>
	0.
$$
Thus, to show the triviality of any non-negative solution of~\eqref{E2.1.1}, \eqref{1.2}, it is sufficient to use Corollary~\ref{C2.2}.

Note that~\eqref{E2.1.2} coincides with the analogous condition obtained in~\cite[Theorem~4.6]{AM}.
\end{Example}

\begin{Example}\label{E2.2}
Let us examine the case of critical exponents $\lambda_1$ and $\lambda_2$ in~\eqref{E2.1.2}.
Namely, consider the system
\begin{equation}
	\left\{
		\begin{aligned}
			&
			- \operatorname{div} A_1 (x, \nabla u_1)
			\ge
			u_2^{\lambda_1}
			\log^{\sigma_1}
			\left(
				e 
				+
				\frac{1}{u_2}
			\right)
			&
			\mbox{in } {\mathbb R}^n,
			\\
			&
			- \operatorname{div} A_2 (x, \nabla u_2)
			\ge
			u_1^{\lambda_2}
			\log^{\sigma_2}
			\left(
				e 
				+
				\frac{1}{u_1}
			\right)
			&
			\mbox{in } {\mathbb R}^n,
		\end{aligned}
	\right.
	\label{E2.2.1}
\end{equation}
where $n > \operatorname{max} \{ p_1, p_2 \}$, the real numbers $\lambda_1, \lambda_2 > 0$ satisfy~\eqref{C2.1.1}, $\sigma_1, \sigma_2 \in {\mathbb R}$ and, moreover,
$$
	\frac{
		n 
		(p_1 - 1) (p_2 - 1)
	}{
		\lambda_1 \lambda_2
	}
	=
	n - p_2 - (p_2 - 1) \frac{p_1}{\lambda_1}
$$
and
$$
	\frac{
		n 
		(p_1 - 1) (p_2 - 1)
	}{
		\lambda_1 \lambda_2
	}
	<
	n - p_1 - (p_1 - 1) \frac{p_2}{\lambda_2}.
$$
If $u_1 = 0$ or $u_2 = 0$, then it is assumed that the right-hand side of the corresponding inequality in~\eqref{E2.2.1} is equal to zero.
It can be verified that for any real numbers $\varepsilon_*, \zeta_* > 0$ there is a real number $r_* > 0$ such that condition~(a) is satisfied with some positive continuous functions
$$
	g_1 (\zeta) 
	\asymp
	\frac{
		\zeta^{
			\lambda_1 \lambda_2 / ((p_1 - 1) (p_2 - 1))
		}
	}{
		\log^{
			|\sigma_2| +  \lambda_2 |\sigma_1| / (p_1 - 1)
		} 
		\zeta
	}
	\quad
	\mbox{and}
	\quad
	q_1 (r)
	\asymp
	\frac{
		1
	}{
		r^n
	}
	\log^{
		\sigma_2 + \lambda_2 \sigma_1 / (p_1 - 1)
	} 
	r.
$$
Thus, in accordance with Theorem~\ref{T2.1} if
$$
	\sigma_2 
	+ 
	\frac{
		\lambda_2 \sigma_1 
	}{
		p_1 - 1
	} 
	\ge 
	- 1,
$$
then any non-negative solution of~\eqref{E2.2.1}, \eqref{1.2} is identically zero.
\end{Example}

\section{Proof of Theorem~\ref{T2.1}}\label{proof}
Let us agree to denote by $C$, $\sigma$, and $\varkappa$ various positive constants that can depend only on $n$, $p_1$, $p_2$, $C_1$, $C_2$, $\varepsilon$, $\theta$, and $\lambda$.
In so doing, by $\chi_\omega$ we mean the characteristic function of a set $\omega \subset {\mathbb R}^n$, i.e.
$$
	\chi_\omega (x)
	=
	\left\{
		\begin{array}{ll}
			1,
			&
			x
			\in
			\omega,
			\\
			0,
			&
			x
			\not\in
			\omega.
		\end{array}
	\right.
$$

Assume that $A$ is a Caratheodory function satisfying the uniform ellipticity condition
\begin{equation}
	C_1
	|\xi|^p
	\le
	\zeta
	A (x, \xi),
	\quad
	|A (x, \xi)|
	\le
	C_2
	|\xi|^p,
	\quad
	p > 1,
	\label{3.0}
\end{equation}
for almost all
$x \in {\mathbb R}^n$
and for all
$\xi \in {\mathbb R}^n$.
We say that $v \in W_{p, loc}^1 (\omega)$ is a solution of the inequality
\begin{equation}
	\operatorname{div} A (x, \nabla v)
	\ge
	a (x)
	\quad
	\mbox{in } \omega,
	\label{3.1}
\end{equation}
where $\omega \subset {\mathbb R}^n$ is a non-empty open set and $a \in L_{1, loc} (\omega)$, if
$$
	-
	\int_\omega
	A (x, \nabla v) \nabla \varphi
	\,
	dx
	\ge
	\int_\omega
	a (x)
	\varphi
	\,
	dx
$$
for any non-negative function $\varphi \in C_0^\infty (\omega)$.
A solution of the inequality
\begin{equation}
	- \operatorname{div} A (x, \nabla u)
	\ge
	a (x)
	\quad
	\mbox{in } \omega,
	\label{3.2}
\end{equation}
where $\omega \subset {\mathbb R}^n$ is a non-empty open set and $a \in L_{1, loc} (\omega)$, is defined in a similar way.

\begin{Lemma}[Generalized Kato's inequality]\label{L3.1}
Let $v \in W_{p, loc}^1 (\omega)$ be a solution of~\eqref{3.1}.
Then $v_+ = \chi_{\omega_+} v$ satisfies the inequality
$$
	\operatorname{div} A (x, \nabla v_+)
	\ge
	\chi_{\omega_+} (x)
	a (x)
	\quad
	\mbox{in } \omega,
$$
where $\omega_+ = \{ x \in \omega : v (x) > 0 \}$. 
\end{Lemma}

The proof of Lemma~\ref{L3.1} is given in~\cite[Lemma~4.2]{meJMAA_2007}.
Putting $v = \varepsilon - u$ in Lemma~\ref{L3.1}, we obtain the following statement.

\begin{Corollary}\label{C3.1}
Let $u$ be a solution of~\eqref{3.2}. Then 
$
	u_\varepsilon 
	= 
	\chi_{\omega_\varepsilon} 
	u 
	+ 
	(1 - \chi_{\omega_\varepsilon})
	\varepsilon
$
satisfies the inequality
$$
	- \operatorname{div} A (x, \nabla u_\varepsilon)
	\ge
	\chi_{\omega_\varepsilon} (x)
	a (x)
	\quad
	\mbox{in } {\mathbb R}^n,
$$
where $\omega_\varepsilon = \{ x \in {\mathbb R}^n : u (x) < \varepsilon \}$.
\end{Corollary}

We also need the three lemmas below.

\begin{Lemma}[Weak Harnack inequality]\label{L3.2}
Let $u \ge 0$ be a solution of the inequality
\begin{equation}
	- \operatorname{div} A (x, \nabla u) \ge 0
	\quad
	\mbox{in } {\mathbb R}^n,
	\label{L3.2.1}
\end{equation}
where  $n > p$. Then 
$$
	\left(
		\frac{
			1
		}{
			\mes B_{2 r}
		}
		\int_{
			B_{2 r}
		}
		u^\lambda
		\,
		dx
	\right)^{1 / \lambda}
	\le
	C
	\operatorname*{ess\,inf}\limits_{
		B_r
	}
	u
$$
for all $\lambda \in (0, n (p - 1) / (n - p))$ and $r \in (0, \infty)$.
\end{Lemma}

\begin{Lemma}\label{L3.3}
Let $u \ge 0$ be a solution of the inequality
\begin{equation}
	- \operatorname{div} A (x, \nabla u)
	\ge
	a (x)
	\quad
	\mbox{in } {\mathbb R}^n,
	\label{L3.3.1}
\end{equation}
where $a \in L_{1, loc} ({\mathbb R}^n)$ is a non-negative function.
If $u^\lambda \in L_{1, loc} ({\mathbb R}^n)$ for some $\lambda \in (p - 1, \infty)$, then
$$
	\frac{1}{\mes B_r}
	\int_{B_r}
	a (x)
	\,
	dx
	\le
	C
	r^{- p}
	\left(
		\frac{
			1
		}{
			\mes B_{2 r} \setminus B_r
		}
		\int_{
			B_{2 r} \setminus B_r
		}
		u^\lambda
		\,
		dx
	\right)^{(p - 1) / \lambda}
$$
for all $r \in (0, \infty)$.
\end{Lemma}

\begin{Lemma}\label{L3.4}
Let $u \ge 0$ be a solution of~\eqref{L3.2.1} such that
\begin{equation}
	\operatorname*{ess\,inf}\limits_{
		{\mathbb R}^n
	}
	u
	=
	0,
	\label{4.8}
\end{equation}
where $n > p$. Then
$$
	\lim_{r \to \infty}
	\frac{
		\mes B_r \setminus \omega_\varepsilon
	}{
		\mes B_r
	}
	=
	0,
$$
where $\omega_\varepsilon = \{ x \in {\mathbb R}^n : u (x) < \varepsilon \}$.
\end{Lemma}

Lemmas~\ref{L3.2} and~\ref{L3.3} are proved in~\cite{Serrin, Trudinger, TW}  and~\cite{AM}, respectively.
Lem\-ma~\ref{L3.4} or its equivalent statement can be found in~\cite[Lemma 3.1]{AMAdvMath}. This lemma is a direct consequence of Lemma~\ref{L3.2}.

Corollary~\ref{C3.1} and Lemmas~\ref{L3.2} and~\ref{L3.3} with $\lambda \in (p - 1, p) \cap (0, n (p - 1) / (n - p))$ imply the following assertion.

\begin{Corollary}\label{C3.2}
Let $u \ge 0$ be a solution of~\eqref{L3.3.1}, where $n > p$ and $a \in L_{1, loc} ({\mathbb R}^n)$ is a non-negative function.
Then 
$$
	\frac{1}{\mes B_r}
	\int_{\omega_\varepsilon \cap B_r}
	a (x)
	\,
	dx
	\le
	C
	r^{- p}
	\left(
		\operatorname*{ess\,inf}\limits_{
			B_r
		}
		u_\varepsilon
	\right)^{p - 1}
$$
for all $r \in (0, \infty)$, 
where 
$\omega_\varepsilon = \{ x \in {\mathbb R}^n : u (x) < \varepsilon \}$ 
and 
$
	u_\varepsilon 
	= 
	\chi_{\omega_\varepsilon} 
	u 
	+ 
	(1 - \chi_{\omega_\varepsilon})
	\varepsilon.
$
\end{Corollary}

\begin{proof}[Proof of Theorem~\ref{T2.1}]
Assume the converse. Let $(u_1, u_2)$ be a non-negative solution of~\eqref{1.1}, \eqref{1.2} that is not equal to zero.
According to Lemma~\ref{L3.2}, both the functions $u_1$ and $u_2$ are positive almost everywhere in ${\mathbb R}^n$. 

We put
$\Omega_\varepsilon = \Omega_{1, \varepsilon} \cap \Omega_{2, \varepsilon}$,
where
$
	\Omega_{k, \varepsilon}
	=
	\{
		x \in {\mathbb R}^n
		:
		u_k (x) < \varepsilon
	\}
$,
$k = 1,2$.
By Lemma~\ref{L3.4}, 
$$
	\lim_{r \to \infty}
	\frac{
		\operatorname{mes} B_r \setminus \Omega_{k, \varepsilon}
	}{
		\operatorname{mes} B_r
	}
	=
	0,
	\quad
	k = 1,2.
$$
Hence, there is a real number $r_0 > 0$ such that
\begin{equation}
	\mes \Omega_\varepsilon \cap B_{\theta r} \setminus B_r
	\ge
	C
	r^n,
	\label{PT2.1.1}
\end{equation}
for all $r \ge r_0$.
Denote
$$
	E_1 (r)
	=
	\int_{
		\Omega_\varepsilon \cap B_r
	}
	F_2 (x, u_1)
	\,
	dx,
	\quad
	r > 0,
$$
and
$$
	E_2 (r)
	=
	\int_{
		\Omega_\varepsilon \cap B_r
	}
	F_1 (x, u_2)
	\,
	dx,
	\quad
	r > 0.
$$

Also let $r_i = \theta^i r_0$, $i = 1,2,\ldots$.
Corollary~\ref{C3.2} implies the estimates
$$
	\frac{
		1
	}{
		\mes B_{r_i}
	}
	\int_{
		\Omega_\varepsilon \cap B_{r_i}
	}
	F_1 (x, u_2)
	\,
	dx
	\le
	C
	r_i^{- p_1}
	\left(
		\operatorname*{ess\,inf}\limits_{
			B_{r_i}
		}
		u_{1, \varepsilon}
	\right)^{p_1 - 1},
	\quad
	i = 1,2,\ldots,
$$
and
$$
	\frac{
		1
	}{
		\mes B_{r_i}
	}
	\int_{
		\Omega_\varepsilon \cap B_{r_i}
	}
	F_2 (x, u_1)
	\,
	dx
	\le
	C
	r_i^{- p_2}
	\left(
		\operatorname*{ess\,inf}\limits_{
			B_{r_i}
		}
		u_{2, \varepsilon}
	\right)^{p_2 - 1},
	\quad
	i = 1,2,\ldots,
$$
where
\begin{equation}
	u_{k, \varepsilon }
	= 
	\chi_{
		\Omega_{k, \varepsilon}
	} 
	u_k 
	+ 
	(
		1 
		- 
		\chi_{
			\Omega_{k, \varepsilon}
		}
	)
	\varepsilon,
	\quad
	k = 1,2,
	\label{PT2.1.8}
\end{equation}
whence it follows that
\begin{equation}
	\operatorname*{ess\,inf}\limits_{
		\Omega_{1, \varepsilon} \cap B_{r_i}
	}
	u_1
	=
	\operatorname*{ess\,inf}\limits_{
		B_{r_i}
	}
	u_{1, \varepsilon}
	\ge
	\sigma
	\left(
		\frac{
			E_2 (r_i)
		}{
			r_i^{n - p_1} 
		}
	\right)^{1 / (p_1 - 1)},
	\quad
	i = 1,2,\ldots,
	\label{PT2.1.2}
\end{equation}
and
\begin{equation}
	\operatorname*{ess\,inf}\limits_{
		\Omega_{2, \varepsilon} \cap B_{r_i}
	}
	u_2
	=
	\operatorname*{ess\,inf}\limits_{
		B_{r_i}
	}
	u_{2, \varepsilon}
	\ge
	\sigma
	\left(
		\frac{
			E_1 (r_i)
		}{
			r_i^{n - p_2} 
		}
	\right)^{1 / (p_2 - 1)},
	\quad
	i = 1,2,\ldots.
	\label{PT2.1.3}
\end{equation}

Taking into account~\eqref{PT2.1.2} and the fact that ${F_1 (x, \cdot)}$ is a non-decreasing function on the interval $[0, \varepsilon]$ for almost $x \in {\mathbb R}^n$, we obtain
$$
	F_2 (x, u_1 (x))
	\ge
	F_2
	\left(
		x,
		\operatorname*{ess\,inf}\limits_{
			\Omega_{1, \varepsilon} \cap B_{r_i}
		}
		u_1
	\right)
	\ge
	F_2
	\left(
		x,
		\sigma
		\left(
			\frac{
				E_2 (r_i)
			}{
				r_i^{n - p_1} 
			}
		\right)^{1 / (p_1 - 1)}
	\right)
$$
for almost all $x \in \Omega_{1, \varepsilon} \cap B_{r_i}$, $i = 1,2,\ldots$.
This immediately implies that
\begin{align*}
	&
	\frac{
		1
	}{
		\operatorname{mes} \Omega_\varepsilon \cap B_{r_i} \setminus B_{r_{i-1}}
	}
	\int_{
		\Omega_\varepsilon \cap B_{r_i} \setminus B_{r_{i-1}}
	}
	F_2 (x, u_1)
	\,
	dx
	\\
	&
	\qquad
	{}
	\ge
	\operatorname*{ess\,inf}\limits_{
		x 
		\in
		\Omega_\varepsilon \cap B_{r_i} \setminus B_{r_{i-1}}
	}
	F_2
	\left(
		x,
		\sigma
		\left(
			\frac{
				E_2 (r_i)
			}{
				r_i^{n - p_1} 
			}
		\right)^{1 / (p_1 - 1)}
	\right),
	\quad
	i = 1,2,\ldots,
\end{align*}
whence in accordance with~\eqref{PT2.1.1} and the definition of $f_2$ we have
\begin{equation}
	E_1 (r_i) - E_1 (r_{i-1})
	\ge
	C
	r^n
	f_2
	\left(
		r,
		\sigma
		\left(
			\frac{
				E_2 (r_i)
			}{
				r_i^{n - p_1} 
			}
		\right)^{1 / (p_1 - 1)}
	\right)
	\label{PT2.1.5}
\end{equation}
for all $r \in (r_{i-1}, r_i)$, $i = 1,2,\ldots$.

Analogously, taking into account~\eqref{PT2.1.1} and~\eqref{PT2.1.3} and the definition of $f_1$, we obtain
\begin{equation}
	E_2 (r_i) - E_2 (r_{i-1})
	\ge
	C
	r^n
	f_1
	\left(
		r,
		\sigma
		\left(
			\frac{
				E_1 (r_i)
			}{
				r_i^{n - p_2} 
			}
		\right)^{1 / (p_2 - 1)}
	\right)
	\label{PT2.1.6}
\end{equation}
for all $r \in (r_{i-1}, r_i)$, $i = 1,2,\ldots$.
The last inequality, in particular, implies that
$$
	E_2 (r_i)
	\ge
	C
	r^n
	f_1
	\left(
		r,
		\sigma
		\left(
			\frac{
				E_1 (r_i)
			}{
				r_i^{n - p_2} 
			}
		\right)^{1 / (p_2 - 1)}
	\right)
$$
for all $r \in (r_{i-1}, r_i)$, $i = 1,2,\ldots$.
Combining this with~\eqref{PT2.1.5}, we arrive at the estimate
$$
	E_1 (r_i) - E_1 (r_{i-1})
	\ge
	C
	r^n
	f_2
	\left(
		r,
		\sigma
		r^{p_1 / (p_1 - 1)}
		f_1^{1 / (p_1 - 1)}
		\left(
			r,
			\left(
				\frac{
					\varkappa
					E_1 (r_i)
				}{
					r^{n - p_2} 
				}
			\right)^{1 / (p_2 - 1)}
		\right)
	\right)
$$
for all $r \in (r_{i-1}, r_i)$, $i = 1,2,\ldots$.
In the same way, it can be shown that
$$
	E_2 (r_i) - E_2 (r_{i-1})
	\ge
	C
	r^n
	f_1
	\left(
		r,
		\sigma
		r^{p_2 / (p_2 - 1)}
		f_2^{1 / (p_2 - 1)}
		\left(
			r,
			\left(
				\frac{
					\varkappa
					E_2 (r_i)
				}{
					r^{n - p_1} 
				}
			\right)^{1 / (p_1 - 1)}
		\right)
	\right)
$$
for all $r \in (r_{i-1}, r_i)$, $i = 1,2,\ldots$.

Thus, there are an integer $i_0 > 0$, non-decreasing functions 
$g_1, g_2 : [\zeta_*, \infty) \to (0, \infty)$, and locally bounded measurable functions 
$q_1, q_2 : [r_*, \infty) \to [0, \infty)$ such that 
$$
	E_s (r_i) - E_s (r_{i-1})
	\ge
	C
	r^n
	q_s (r)
	g_s (\varkappa E_s (r_i)),
	\quad
	s = 1, 2,
$$
for all $i > i_0$ and $r \in (r_{i-1}, r_i)$ and, moreover, conditions~\eqref{T2.1.1} and~\eqref{T2.1.2} are valid with some $k \in \{ 1, 2 \}$, where $r_* = r_{i_0}$ and $\zeta_* = \varkappa \operatorname{min} \{ E_1 (r_{i_0}), E_2 (r_{i_0}) \}$.
We obviously have
$$
	\frac{
		E_k (r_i) - E_k (r_{i-1})
	}{
		g_k (\varkappa E_k (r_i)),
	}
	\ge
	C
	r^n
	q_k (r)
$$
for all $i > i_0$ and $r \in (r_{i-1}, r_i)$, whence in accordance with the inequalities
$$
	\int_{
		E_k (r_{i-1})
	}^{
		E_k (r_i)
	}
	\frac{
		d\zeta
	}{
		g_k (\varkappa \zeta))
	}
	\ge
	\frac{
		E_k (r_i) - E_k (r_{i-1})
	}{
		g_k (\varkappa E_k (r_i)),
	}
$$
and
$$
	\sup_{
		r \in (r_{i-1}, r_i)
	}
	r^n
	q_k (r)
	\ge
	C
	\int_{
		r_{i-1}
	}^{
		r_i
	}
	r^{n - 1}
	q_k (r)
	\,
	dr
$$
it follows that
$$
	\int_{
		E_k (r_{i-1})
	}^{
		E_k (r_i)
	}
	\frac{
		d\zeta
	}{
		g_k (\varkappa \zeta))
	}
	\ge
	C
	\int_{
		r_{i-1}
	}^{
		r_i
	}
	r^{n - 1}
	q_k (r)
	\,
	dr,
	\quad
	i = i_0 + 1, i_0 + 2, \ldots.
$$
Summing the last expression over all $i > i_0$, we obtain
$$
	\int_{
		E_k (r_{i_0})
	}^\infty
	\frac{
		d\zeta
	}{
		g_k (\varkappa \zeta))
	}
	\ge
	C
	\int_{
		r_{i_0}
	}^\infty
	r^{n - 1}
	q_k (r)
	\,
	dr.
$$
This contradicts~\eqref{T2.1.1} and~\eqref{T2.1.2}.
\end{proof}

\section{Some remarks and generalizations}

As in the previous section, we denote by $C$, $\sigma$, and $\varkappa$ various positive constants that can depend only on $n$, $p_1$, $p_2$, $C_1$, $C_2$, $\varepsilon$, $\theta$, and $\lambda$.

\paragraph{4.1}
Consider systems of the form
$$
	\left\{
		\begin{aligned}
			&
			- \operatorname{div} A_1 (x, \nabla u_1)
			\ge
			a_1 (x) h_1 (u_2)
			&
			\mbox{in } {\mathbb R}^n,
			\\
			&
			- \operatorname{div} A_2 (x, \nabla u_2)
			\ge
			a_2 (x) h_2 (u_1)
			&
			\mbox{in } {\mathbb R}^n,
		\end{aligned}
	\right.
$$
where $a_1, a_2 : {\mathbb R}^n \to (0, \infty)$ and $h_1, h_2 : [0, \infty) \to [0, \infty)$ are measurable functions such that $h_1$ and $h_2$ are positive on the interval $(0, \infty)$ and non-decreasing on $[0, \varepsilon]$ for some real number $0 < \varepsilon < 1$. 
It does not present any particular problem to verify that for these systems Theorem~\ref{T2.1} remains valid with~\eqref{1.4} replaced by
\begin{equation}
	f_k (r, \zeta)
	=
	\frac{
		h_k (\zeta)
	}{
		\left(
			\frac{
				1
			}{
				\operatorname{mes} B_{\theta r} \setminus B_{r / \theta}
			}
			\int_{
				B_{\theta r} \setminus B_{r / \theta}
			}
			a_k^{- \lambda} (x)
			\,
			dx
		\right)^{1 / \lambda}
	},
	\quad
	r, \zeta > 0,
	\;
	k = 1, 2,
	\label{4.1}
\end{equation}
where $\lambda > 0$ is a real number.
In so doing, if
$$
	\int_{
		B_{\theta r} \setminus B_{r / \theta}
	}
	a_k^{- \lambda} (x)
	\,
	dx
	=
	\infty
$$
for some $r \in (0, \infty)$ and $k \in \{ 1, 2 \}$, then we assume that
$
	f_k (r, \zeta) = 0
$
for all $\zeta \in (0, \infty)$.

Indeed, inequality~\eqref{PT2.1.2} yields
$$
	h_2 (u_1 (x))
	\ge
	h_2
	\left(
		\sigma
		\left(
			\frac{
				E_2 (r_i)
			}{
				r_i^{n - p_1} 
			}
		\right)^{1 / (p_1 - 1)}
	\right)
$$
for almost all $x \in \Omega_{1, \varepsilon} \cap B_{r_i}$, whence it follows that
\begin{align*}
	&
	\frac{
		1
	}{
		\operatorname{mes} \Omega_\varepsilon \cap B_{r_i} \setminus B_{r_{i-1}}
	}
	\int_{
		\Omega_\varepsilon \cap B_{r_i} \setminus B_{r_{i-1}}
	}
	h_2^{\lambda / (1 + \lambda)} (u_1)
	\,
	dx
	\\
	&
	\qquad
	{}
	\ge
	h_2^{\lambda / (1 + \lambda)}
	\left(
		\sigma
		\left(
			\frac{
				E_2 (r_i)
			}{
				r_i^{n - p_1} 
			}
		\right)^{1 / (p_1 - 1)}
	\right),
	\quad
	i = 1,2,\ldots.
\end{align*}
Combining this with~\eqref{PT2.1.1} and the estimate
\begin{align*}
	&
	\int_{
		\Omega_\varepsilon \cap B_{r_i} \setminus B_{r_{i-1}}
	}
	h_2^{\lambda / (1 + \lambda)} (u_1)
	\,
	dx
	=
	\int_{
		\Omega_\varepsilon \cap B_{r_i} \setminus B_{r_{i-1}}
	}
	a_2^{- \lambda / (1 + \lambda)} (x)
	a_2^{\lambda / (1 + \lambda)} (x)
	h_2^{\lambda / (1 + \lambda)} (u_1)
	\,
	dx
	\\
	&
	\quad
	{}
	\le
	\left(
		\int_{
			\Omega_\varepsilon \cap B_{r_i} \setminus B_{r_{i-1}}
		}
		a_2^{- \lambda} (x)
		\,
		dx
	\right)^{1 / (1 + \lambda)}
	\left(
		\int_{
			\Omega_\varepsilon \cap B_{r_i} \setminus B_{r_{i-1}}
		}
		a_2 (x)
		h_2 (u_1)
		\,
		dx
	\right)^{\lambda / (1 + \lambda)}
\end{align*}
which follows from H\"older's inequality, we obtain
$$
	\int_{
		\Omega_\varepsilon \cap B_{r_i} \setminus B_{r_{i-1}}
	}
	a_2 (x)
	h_2 (u_1)
	\,
	dx
	\ge
	\frac{
		C
		r_i^n
		h_2
		\left(
			\sigma
			\left(
				\frac{
					E_2 (r_i)
				}{
					r_i^{n - p_1} 
				}
			\right)^{1 / (p_1 - 1)}
		\right)
	}{
		\left(
			\frac{
				1
			}{
				\operatorname{mes} B_{r_i} \setminus B_{r_{i-1}}
			}
			\int_{
				B_{r_i} \setminus B_{r_{i-1}}
			}
			a_2^{- \lambda} (x)
			\,
			dx
		\right)^{1 / \lambda}
	},
	\quad
	i = 1,2,\ldots.
$$
The last formula obviously implies~\eqref{PT2.1.5} with
$$
	E_1 (r)
	=
	\int_{
		\Omega_\varepsilon \cap B_r
	}
	a_2 (x)
	h_2 (u_1)
	\,
	dx,
	\quad
	r > 0,
$$
and $f_2$ defined by~\eqref{4.1}.

Repeating the previous reasoning with~\eqref{PT2.1.2} replaced by~\eqref{PT2.1.3}, we also have
$$
	\int_{
		\Omega_\varepsilon \cap B_{r_i} \setminus B_{r_{i-1}}
	}
	a_1 (x)
	h_1 (u_2)
	\,
	dx
	\ge
	\frac{
		C
		r_i^n
		h_1
		\left(
			\sigma
			\left(
				\frac{
					E_2 (r_i)
				}{
					r_i^{n - p_2} 
				}
			\right)^{1 / (p_2 - 1)}
		\right)
	}{
		\left(
			\frac{
				1
			}{
				\operatorname{mes} B_{r_i} \setminus B_{r_{i-1}}
			}
			\int_{
				B_{r_i} \setminus B_{r_{i-1}}
			}
			a_1^{- \lambda} (x)
			\,
			dx
		\right)^{1 / \lambda}
	},
	\quad
	i = 1,2,\ldots.
$$
This in turn implies~\eqref{PT2.1.6} with
$$
	E_2 (r)
	=
	\int_{
		\Omega_\varepsilon \cap B_r
	}
	a_1 (x)
	h_1 (u_2)
	\,
	dx,
	\quad
	r > 0,
$$
and $f_1$ defined by~\eqref{4.1}.

\paragraph{4.2}
An interesting case is when $F_1$ and $F_2$ on the right in~\eqref{1.1} are non-decreasing functions with respect to the last argument on the whole interval $[0, \infty)$. Fix some real number $0 < \varepsilon < 1$. Also let for any real numbers $\varepsilon_*, \zeta_* > 0$ there be a real number $r_* > 0$ and functions
$q_k : {\mathbb R}^n \setminus B_{r_*} \to [0, \infty)$
and
$g_k : [\zeta_*, \infty) \to (0, \infty)$, $k = 1,2$,
satisfying the following conditions:
\begin{enumerate}
\item[$(a')$]
for almost all $x \in {\mathbb R}^n \setminus B_{r_*}$ and for all $\zeta \in [\zeta_*, \infty)$ such that
$$
	\left(
		\frac{
			\zeta
		}{
			|x|^{n - p_2}
		}
	\right)^{1 / (p_2 - 1)}
	\le
	\varepsilon
	\quad
	\mbox{and}
	\quad
	\varepsilon_*
	|x|^{p_1 / (p_1 - 1)}
	f_1^{1 / (p_1 - 1)}
	\left(
		|x|, 
		\left(
			\frac{
				\zeta
			}{
				|x|^{n - p_2}
			}
		\right)^{1 / (p_2 - 1)}
	\right)
	\le
	\varepsilon
$$
we have
\begin{equation}
	F_2 
	\left( 
		x,
		\varepsilon_*
		|x|^{p_1 / (p_1 - 1)}
		f_1^{1 / (p_1 - 1)}
		\left(
			|x|,
			\left(
				\frac{
					\zeta
				}{
					|x|^{n - p_2}
				}
			\right)^{1 / (p_2 - 1)}
		\right)
	\right)
	\ge
	q_1 (x) 
	g_1 (\zeta);
	\label{1.5new}
\end{equation}
\item[$(b')$]
for almost all $x \in {\mathbb R}^n \setminus B_{r_*}$ and for all $\zeta \in [\zeta_*, \infty)$ such that
$$
	\left(
		\frac{
			\zeta
		}{
			|x|^{n - p_1}
		}
	\right)^{1 / (p_1 - 1)}
	\le
	\varepsilon
	\quad
	\mbox{and}
	\quad
	\varepsilon_*
	|x|^{p_2 / (p_2 - 1)}
	f_2^{1 / (p_2 - 1)}
	\left(
		|x|,
		\left(
			\frac{
				\zeta
			}{
				|x|^{n - p_1}
			}
		\right)^{1 / (p_1 - 1)}
	\right)
	\le
	\varepsilon
$$
we have
\begin{equation}
	F_1
	\left(
		x,
		\varepsilon_*
		|x|^{p_2 / (p_2 - 1)}
		f_2^{1 / (p_2 - 1)}
		\left(
			|x|,
			\left(
				\frac{
					\zeta
				}{
					|x|^{n - p_1}
				}
			\right)^{1 / (p_1 - 1)}
		\right)
	\right)
	\ge
	q_2 (x) 
	g_2 (\zeta),
	\label{1.6new}
\end{equation}
\end{enumerate}
where
$$
	f_k (r, \zeta)
	=
	\frac{
		1
	}{
		\operatorname{mes} B_r
	}
	\int_{B_r}
	F_k (x, \zeta)
	\,
	dx,
	\quad
	r, \zeta > 0,
	\;
	k = 1, 2.
$$

\begin{Theorem}\label{T4.1}
Suppose that $n > \operatorname{max} \{ p_1, p_2 \}$ and, moreover, the functions $F_1$ and $F_2$ are non-negative and non-decreasing with respect to the last arguments on the set ${\mathbb R}^n \times [0, \infty)$ and positive on ${\mathbb R}^n \times (0, \infty)$.
Also let for any real numbers $\varepsilon_*, \zeta_* > 0$ there be a real number $r_* > 0$, non-decreasing functions 
$g_1, g_2 : [\zeta_*, \infty) \to (0, \infty)$, and locally bounded measurable functions 
$q_1, q_2 : {\mathbb R}^n \setminus B_{r_*} \to [0, \infty)$ such that $(a')$ and $(b')$ are valid. 
If there exist $k \in \{ 1, 2 \}$ such that~\eqref{T2.1.1} holds and 
\begin{equation}
	\int_{
		{\mathbb R}^n \setminus B_{r_*}
	}
	q_k (x)
	\,
	dx
	=
	\infty,
	\label{T4.1.1}
\end{equation}
then any non-negative solution of~\eqref{1.1}, \eqref{1.2} is identically zero.
\end{Theorem}

\begin{proof}
Arguing by contradiction, we assume that $(u_1, u_2)$ is a non-negative solution of~\eqref{1.1}, \eqref{1.2} that is not equal to zero.
By Lemma~\ref{L3.2}, both the functions $u_1$ and $u_2$ are positive almost everywhere in ${\mathbb R}^n$. 
We take a real number $r_0 > 0$ such that
\begin{equation}
	\operatorname*{ess\,inf}\limits_{
		B_{r_0}
	}
	u_k
	<
	\varepsilon,
	\quad
	k = 1, 2.
	\label{PT4.1.1}
\end{equation}
In view of~\eqref{1.2}, such a real number $r_0$ obviously exists.
Let us put $r_i = 2^i$, $i = 1,2,\ldots$.

Applying Lemmas~\ref{L3.2} and~\ref{L3.3} with $\lambda \in (p - 1, p) \cap (0, n (p - 1) / (n - p))$, we obtain
$$
	\frac{
		1
	}{
		\mes B_{r_i}
	}
	\int_{
		B_{r_i}
	}
	F_1 (x, u_2)
	\,
	dx
	\le
	C
	r_i^{- p_1}
	\left(
		\operatorname*{ess\,inf}\limits_{
			B_{r_i}
		}
		u_1
	\right)^{p_1 - 1},
	\quad
	i = 1,2,\ldots,
$$
and
$$
	\frac{
		1
	}{
		\mes B_{r_i}
	}
	\int_{
		B_{r_i}
	}
	F_2 (x, u_1)
	\,
	dx
	\le
	C
	r_i^{- p_2}
	\left(
		\operatorname*{ess\,inf}\limits_{
			B_{r_i}
		}
		u_2
	\right)^{p_2 - 1},
	\quad
	i = 1,2,\ldots,
$$
whence it follows that
\begin{equation}
	\operatorname*{ess\,inf}\limits_{
		B_{r_i}
	}
	u_1
	\ge
	\sigma
	\left(
		\frac{
			E_2 (r_i)
		}{
			r_i^{n - p_1} 
		}
	\right)^{1 / (p_1 - 1)},
	\quad
	i = 1,2,\ldots,
	\label{4.3}
\end{equation}
and
\begin{equation}
	\operatorname*{ess\,inf}\limits_{
		B_{r_i}
	}
	u_2
	\ge
	\sigma
	\left(
		\frac{
			E_1 (r_i)
		}{
			r_i^{n - p_2} 
		}
	\right)^{1 / (p_2 - 1)},
	\quad
	i = 1,2,\ldots,
	\label{4.4}
\end{equation}
where
$$
	E_1 (r)
	=
	\int_{
		B_r
	}
	F_2 (x, u_1)
	\,
	dx,
	\quad
	r > 0,
$$
and
$$
	E_2 (r)
	=
	\int_{
		B_r
	}
	F_1 (x, u_2)
	\,
	dx,
	\quad
	r > 0.
$$

Inequalities~\eqref{4.3} and~\eqref{4.4} imply the estimates
$$
	F_2 (x, u_1 (x))
	\ge
	F_2
	\left(
		x,
		\operatorname*{ess\,inf}\limits_{
			B_{r_i}
		}
		u_1
	\right)
	\ge
	F_2
	\left(
		x,
		\sigma
		\left(
			\frac{
				E_2 (r_i)
			}{
				r_i^{n - p_1} 
			}
		\right)^{1 / (p_1 - 1)}
	\right)
$$
and
$$
	F_1 (x, u_2 (x))
	\ge
	F_1
	\left(
		x,
		\operatorname*{ess\,inf}\limits_{
			B_{r_i}
		}
		u_2
	\right)
	\ge
	F_1
	\left(
		x,
		\sigma
		\left(
			\frac{
				E_1 (r_i)
			}{
				r_i^{n - p_2} 
			}
		\right)^{1 / (p_2 - 1)}
	\right)
$$
for almost all $x \in B_{r_i}$, $i = 1,2,\ldots$.
Integrating them over $B_{r_i}$, we have
$$
	E_1 (r_i)
	\ge
	\operatorname{mes} B_r
	f_2
	\left(
		r_i,
		\sigma
		\left(
			\frac{
				E_2 (r_i)
			}{
				r_i^{n - p_1} 
			}
		\right)^{1 / (p_1 - 1)}
	\right),
	\quad
	i = 1,2,\ldots,
$$
and
$$
	E_2 (r_i)
	\ge
	\operatorname{mes} B_r
	f_1
	\left(
		r_i,
		\sigma
		\left(
			\frac{
				E_1 (r_i)
			}{
				r_i^{n - p_2} 
			}
		\right)^{1 / (p_2 - 1)}
	\right),
	\quad
	i = 1,2,\ldots.
$$
Consequently, one can assert that
$$
	F_2 (x, u_1 (x))
	\ge
	F_2
	\left(
		x,
		\sigma
		r_i^{p_1 / (p_1 - 1)} 
		f_1^{1 / (p_1 - 1)}
		\left(
			r_i,
			\left(
				\frac{
					\varkappa
					E_1 (r_i)
				}{
					r_i^{n - p_2} 
				}
			\right)^{1 / (p_2 - 1)}
		\right)
	\right)
$$
and
$$
	F_1 (x, u_2 (x))
	\ge
	F_1
	\left(
		x,
		\sigma
		r_i^{p_2 / (p_2 - 1)} 
		f_2^{1 / (p_2 - 1)}
		\left(
			r_i,
			\left(
				\frac{
					\varkappa
					E_2 (r_i)
				}{
					r_i^{n - p_1} 
				}
			\right)^{1 / (p_1 - 1)}
		\right)
	\right)
$$
for almost all $x \in B_{r_i}$, $i = 1,2,\ldots$.
In view of conditions~$(a')$ and~$(b')$, this yields
$$
	F_2 (x, u_1 (x))
	\ge
	q_1 (x) 
	g_1 (\varkappa E_1 (r_i))
$$
and
$$
	F_1 (x, u_2 (x))
	\ge
	q_2 (x) 
	g_2 (\varkappa E_2 (r_i))
$$
for almost all $x \in B_{r_i} \setminus B_{r_{i-1}}$, $i = i_0 + 1, i_0 + 2,\ldots$,
where $i_0 \ge 0$ is some integer.
Integrating further the last two inequalities over $B_{r_i} \setminus B_{r_{i-1}}$, we obtain
\begin{equation}
	E_1 (r_i) - E_1 (r_{i-1})
	\ge
	g_1 (\varkappa E_1 (r_i))
	\int_{
		B_{r_i} \setminus B_{r_{i-1}}
	}
	q_1 (x)
	\,
	dx,
	\quad
	i = i_0 + 1, i_0 + 2, \ldots,
	\label{4.5}
\end{equation}
and
\begin{equation}
	E_2 (r_i) - E_2 (r_{i-1})
	\ge
	g_2 (\varkappa E_2 (r_i))
	\int_{
		B_{r_i} \setminus B_{r_{i-1}}
	}
	q_2 (x)
	\,
	dx,
	\quad
	i = i_0 + 1, i_0 + 2, \ldots.
	\label{4.6}
\end{equation}

In can be assumed without loss of generality that~\eqref{T2.1.1} and~\eqref{T4.1.1} are valid for ${k = 1}$. It follows from~\eqref{4.5} that
$$
	\sum_{i = i_0 + 1}^\infty
	\frac{
		E_1 (r_i) - E_1 (r_{i-1})
	}{
		g_1 (\varkappa E_1 (r_i))
	}
	\ge
	\int_{
		{\mathbb R}^n \setminus B_{r_{i_0}}
	}
	q_1 (x)
	\,
	dx,
$$
whence in accordance with the evident estimates
$$
	\int_{
		E_1 (r_{i-1})
	}^{
		E_1 (r_i)
	}
	\frac{
		d\zeta
	}{
		g_1 (\varkappa \zeta)
	}
	\ge
	\frac{
		E_1 (r_i) - E_1 (r_{i-1})
	}{
		g_1 (\varkappa E_1 (r_i))
	},
	\quad
	i = i_0 + 1, i_0 + 2, \ldots,
$$
we have
$$
	\int_{
		E_1 (r_{i_0})
	}^\infty
	\frac{
		d\zeta
	}{
		g_1 (\varkappa \zeta)
	}
	\ge
	\int_{
		{\mathbb R}^n \setminus B_{r_{i_0}}
	}
	q_1 (x)
	\,
	dx.
$$
This contradicts~\eqref{T2.1.1} and~\eqref{T4.1.1}. 

To complete the proof, it remains to note that, in the case where~\eqref{T2.1.1} and~\eqref{T4.1.1} are valid for $k = 2$, we repeat our reasoning with~\eqref{4.5} replaced by~\eqref{4.6}.
\end{proof}

\paragraph{4.3}
The above method is suitable for non-negative solutions of the inequality
\begin{equation}
	- \operatorname{div} A (x, \nabla u) 
	\ge 
	F (x, u)
	\quad
	\mbox{in } {\mathbb R}^n
	\label{4.7}
\end{equation}
for which~\eqref{4.8} is valid,
where $A$ is a Caratheodory function satisfying the uniform ellipticity condition~\eqref{3.0} 
and the function $F$ is non-negative on ${\mathbb R}^n \times [0, \varepsilon]$, positive on ${\mathbb R}^n \times (0, \varepsilon)$, and non-decreasing with respect to the last argument on the interval $[0, \varepsilon]$ for some real number $0 < \varepsilon < 1$.

Denote
\begin{equation}
	f (r, \zeta)
	=
	\operatorname*{ess\,inf}\limits_{
		x \in B_{\theta r} \setminus B_{r / \theta}
	}
	F (x, \zeta),
	\quad
	r, \zeta > 0,
	\label{4.9}
\end{equation}
where $\theta > 1$ is some given real number.
We shall assume that for any real number $\zeta_* > 0$ there exist
a real number $r_* > 0$ and functions
$q : [r_*, \infty) \to [0, \infty)$
and
$g : [\zeta_*, \infty) \to (0, \infty)$
satisfying the following condition:
\begin{enumerate}
\item[$(c)$]
for all $r \in [r_*, \infty)$ and $\zeta \in [\zeta_*, \infty)$ such that
$$
	\left(
		\frac{
			\zeta
		}{
			r^{n - p}
		}
	\right)^{1 / (p - 1)}
	\le
	\varepsilon
$$
we have
$$
	f 
	\left( 
		r,
		\left(
			\frac{
				\zeta
			}{
				r^{n - p}
			}
		\right)^{1 / (p - 1)}
	\right)
	\ge
	q (r) 
	g (\zeta).
$$
\end{enumerate}

\begin{Theorem}\label{T4.2}
Suppose that $n > p$. 
Also let for any real number $\zeta_* > 0$ there exist a real number $r_* > 0$, a locally bounded measurable function
$q : [r_*, \infty) \to [0, \infty)$,
and a non-decreasing function 
$g : [\zeta_*, \infty) \to (0, \infty)$
satisfying condition~$(c)$ and, moreover,
\begin{equation}
	\int_{\zeta_*}^\infty
	\frac{
		d\zeta
	}{
		g (\zeta)
	}
	<
	\infty
	\label{T2.4.1}
\end{equation}
and
\begin{equation}
	\int_{r_*}^\infty
	r^{n - 1}
	q (r)
	\,
	dr
	=
	\infty.
	\label{T2.4.2}
\end{equation}
Then any non-negative solution of~\eqref{4.7}, \eqref{4.8} is identically zero.
\end{Theorem}

\begin{proof}
By contradiction, we assume that~\eqref{4.7}, \eqref{4.8} has a solution $u \ge 0$ that is not identically zero. According to Lemma~\ref{3.2}, this solution is positive almost everywhere in ${\mathbb R}^n$.
We denote 
$
	u_\varepsilon 
	= 
	\chi_{\Omega_\varepsilon} 
	u 
	+ 
	(1 - \chi_{\Omega_\varepsilon})
	\varepsilon,
$
where
$
	\Omega_\varepsilon 
	= 
	\{ 
		x \in {\mathbb R}^n 
		: 
		u (x) < \varepsilon 
	\}
$ 
and $\chi_{\Omega_\varepsilon}$ is the characteristic function of the set $\Omega_\varepsilon$. 
Also let
$$
	E (r)
	=
	\int_{
		\Omega_\varepsilon \cap B_r
	}
	F (x, u)
	\,
	dx,
	\quad
	r > 0.
$$

By Lemma~\ref{L3.4} there is a real number $r_0 > 0$ such that~\eqref{PT2.1.1} is valid for all $r \ge r_0$. As in the proof of Theorem~\ref{T2.1}, we put $r_i = \theta^i r_0$, $i = 1,2,\ldots$. Corollary~\ref{C3.2} yields
$$
	\frac{
		1
	}{
		\mes B_{r_i}
	}
	\int_{
		\Omega_\varepsilon \cap B_{r_i}
	}
	F (x, u)
	\,
	dx
	\le
	C
	r_i^{- p}
	\left(
		\operatorname*{ess\,inf}\limits_{
			B_{r_i}
		}
		u_\varepsilon
	\right)^{p - 1},
	\quad
	i = 1,2,\ldots,
$$
whence it follows that
\begin{equation}
	\operatorname*{ess\,inf}\limits_{
		\Omega_\varepsilon \cap B_{r_i}
	}
	u
	=
	\operatorname*{ess\,inf}\limits_{
		B_{r_i}
	}
	u_\varepsilon
	\ge
	\sigma
	\left(
		\frac{
			E (r_i)
		}{
			r_i^{n - p} 
		}
	\right)^{1 / (p - 1)},
	\quad
	i = 1,2,\ldots.
	\label{PT4.2.2}
\end{equation}
This in turn implies the estimate
$$
	F (x, u (x))
	\ge
	F
	\left(
		x,
		\operatorname*{ess\,inf}\limits_{
			\Omega_\varepsilon \cap B_{r_i}
		}
		u
	\right)
	\ge
	F
	\left(
		x,
		\sigma
		\left(
			\frac{
				E (r_i)
			}{
				r_i^{n - p} 
			}
		\right)^{1 / (p - 1)}
	\right)
$$
for almost all $x \in \Omega_\varepsilon \cap B_{r_i}$, $i = 1,2,\ldots$.
Consequently, 
\begin{align*}
	&
	\frac{
		1
	}{
		\operatorname{mes} \Omega_\varepsilon \cap B_{r_i} \setminus B_{r_{i-1}}
	}
	\int_{
		\Omega_\varepsilon \cap B_{r_i} \setminus B_{r_{i-1}}
	}
	F (x, u)
	\,
	dx
	\\
	&
	\qquad
	{}
	\ge
	\operatorname*{ess\,inf}\limits_{
		x 
		\in
		\Omega_\varepsilon \cap B_{r_i} \setminus B_{r_{i-1}}
	}
	F
	\left(
		x,
		\sigma
		\left(
			\frac{
				E (r_i)
			}{
				r_i^{n - p} 
			}
		\right)^{1 / (p - 1)}
	\right),
	\quad
	i = 1,2,\ldots,
\end{align*}
whence in accordance with~\eqref{PT2.1.1} and~\eqref{4.9} we have
$$
	E (r_i) - E (r_{i-1})
	\ge
	C
	r^n
	f
	\left(
		r,
		\left(
			\frac{
				\varkappa
				E (r_i)
			}{
				r^{n - p} 
			}
		\right)^{1 / (p - 1)}
	\right)
$$
for all $r \in (r_{i-1}, r_i)$, $i = 1,2,\ldots$.
By condition~$(c)$, this implies that
\begin{equation}
	\frac{
		E (r_i) - E (r_{i-1})
	}{
		g (\varkappa E (r_i))
	}
	\ge
	C
	r^n
	q (r)
	\label{PT2.2.3}
\end{equation}
for all $r \in (r_{i-1}, r_i)$, $i = i_0 + 1, i_0 + 2,\ldots$,
where $i_0 \ge 0$ is some integer.
Combining the last estimate with the inequalities
\begin{equation}
	\int_{
		E (r_{i-1})
	}^{
		E (r_i)
	}
	\frac{
		d\zeta
	}{
		g (\varkappa \zeta))
	}
	\ge
	\frac{
		E (r_i) - E (r_{i-1})
	}{
		g (\varkappa E (r_i))
	}
	\label{PT4.2.1}
\end{equation}
and
$$
	\sup_{
		r \in (r_{i-1}, r_i)
	}
	r^n
	q (r)
	\ge
	C
	\int_{
		r_{i-1}
	}^{
		r_i
	}
	r^{n - 1}
	q (r)
	\,
	dr,
$$
we obtain
$$
	\int_{
		E (r_{i-1})
	}^{
		E (r_i)
	}
	\frac{
		d\zeta
	}{
		g (\varkappa \zeta))
	}
	\ge
	C
	\int_{
		r_{i-1}
	}^{
		r_i
	}
	r^{n - 1}
	q (r)
	\,
	dr,
	\quad
	i = i_0 + 1, i_0 + 2,\ldots.
$$
Summing the last expression over all integers $i > i_0$, one can conclude that
$$
	\int_{
		E (r_{i_0})
	}^\infty
	\frac{
		d\zeta
	}{
		g (\varkappa \zeta))
	}
	\ge
	C
	\int_{
		r_{i_0}
	}^\infty
	r^{n - 1}
	q (r)
	\,
	dr.
$$
Thus, we arrive at a contradiction with~\eqref{T2.4.1} and~\eqref{T2.4.2}.
\end{proof}

For $F (x, \zeta) = a (x) h (\zeta)$, Theorem~\ref{T4.2} remains valid with~\eqref{4.9} replaced by
\begin{equation}
	f (r, \zeta)
	=
	\frac{
		h (\zeta)
	}{
		\left(
			\frac{
				1
			}{
				\operatorname{mes} B_{\theta r} \setminus B_{r / \theta}
			}
			\int_{
				B_{\theta r} \setminus B_{r / \theta}
			}
			a^{- \lambda} (x)
			\,
			dx
		\right)^{1 / \lambda}
	},
	\quad
	r, \zeta > 0,
	\label{4.9new}
\end{equation}
where $\lambda > 0$ is some real number. 
To see this, it suffices to repeat the reasoning given in paragraph~4.1. 
Really, from~\eqref{PT4.2.2}, it follows that
$$
	h (u (x))
	\ge
	h
	\left(
		\sigma
		\left(
			\frac{
				E (r_i)
			}{
				r_i^{n - p} 
			}
		\right)^{1 / (p - 1)}
	\right)
$$
for almost all $x \in \Omega_\varepsilon \cap B_{r_i}$, $i = 1,2,\ldots$. Consequently, we have
\begin{align*}
	&
	\frac{
		1
	}{
		\operatorname{mes} \Omega_\varepsilon \cap B_{r_i} \setminus B_{r_{i-1}}
	}
	\int_{
		\Omega_\varepsilon \cap B_{r_i} \setminus B_{r_{i-1}}
	}
	h^{\lambda / (1 + \lambda)} (u)
	\,
	dx
	\\
	&
	\qquad
	{}
	\ge
	h^{\lambda / (1 + \lambda)}
	\left(
		\sigma
		\left(
			\frac{
				E (r_i)
			}{
				r_i^{n - p} 
			}
		\right)^{1 / (p - 1)}
	\right),
	\quad
	i = 1,2,\ldots.
\end{align*}
Combining this with~\eqref{PT2.1.1} and the estimate
\begin{align*}
	&
	\int_{
		\Omega_\varepsilon \cap B_{r_i} \setminus B_{r_{i-1}}
	}
	h^{\lambda / (1 + \lambda)} (u)
	\,
	dx
	=
	\int_{
		\Omega_\varepsilon \cap B_{r_i} \setminus B_{r_{i-1}}
	}
	a^{- \lambda / (1 + \lambda)} (x)
	a^{\lambda / (1 + \lambda)} (x)
	h^{\lambda / (1 + \lambda)} (u)
	\,
	dx
	\\
	&
	\quad
	{}
	\le
	\left(
		\int_{
			\Omega_\varepsilon \cap B_{r_i} \setminus B_{r_{i-1}}
		}
		a^{- \lambda} (x)
		\,
		dx
	\right)^{1 / (1 + \lambda)}
	\left(
		\int_{
			\Omega_\varepsilon \cap B_{r_i} \setminus B_{r_{i-1}}
		}
		a (x)
		h (u)
		\,
		dx
	\right)^{\lambda / (1 + \lambda)}
\end{align*}
which follows from H\"older's inequality, we obtain
$$
	\int_{
		\Omega_\varepsilon \cap B_{r_i} \setminus B_{r_{i-1}}
	}
	a (x)
	h (u)
	\,
	dx
	\ge
	\frac{
		C
		r_i^n
		h
		\left(
			\sigma
			\left(
				\frac{
					E (r_i)
				}{
					r_i^{n - p} 
				}
			\right)^{1 / (p - 1)}
		\right)
	}{
		\left(
			\frac{
				1
			}{
				\operatorname{mes} B_{r_i} \setminus B_{r_{i-1}}
			}
			\int_{
				B_{r_i} \setminus B_{r_{i-1}}
			}
			a^{- \lambda} (x)
			\,
			dx
		\right)^{1 / \lambda}
	},
	\quad
	i = 1,2,\ldots.
$$
The last relation implies~\eqref{PT2.2.3} with
$$
	E (r)
	=
	\int_{
		\Omega_\varepsilon \cap B_r
	}
	a (x)
	h (u_1)
	\,
	dx,
	\quad
	r > 0,
$$
and $f$ defined by~\eqref{4.9new}.

Theorem~\ref{T4.2} can be strengthen if we assume that the function $F$ is non-negative and non-decreasing with respect to the last argument on the whole set ${\mathbb R}^n \times [0, \infty)$
and positive on ${\mathbb R}^n \times (0, \infty)$.
Namely, let for any real number $\zeta_* > 0$ there be a real number $r_* > 0$ and functions
$q : {\mathbb R}^n \setminus B_{r_*} \to [0, \infty)$
and
$g : [\zeta_*, \infty) \to (0, \infty)$
satisfying the following condition:
\begin{enumerate}
\item[$(c')$]
for almost all $x \in {\mathbb R}^n \setminus B_{r_*}$ and for all $\zeta \in [\zeta_*, \infty)$ such that
$$
	\left(
		\frac{
			\zeta
		}{
			|x|^{n - p}
		}
	\right)^{1 / (p - 1)}
	\le
	\varepsilon
$$
we have
$$
	F 
	\left( 
		x,
		\left(
			\frac{
				\zeta
			}{
				|x|^{n - p}
			}
		\right)^{1 / (p - 1)}
	\right)
	\ge
	q (x) 
	g (\zeta).
$$
\end{enumerate}
Here $\varepsilon \in (0, 1)$ can be an arbitrary real number. We should only assume that $\varepsilon$ does not depend on $\zeta_*$.

\begin{Theorem}\label{T4.3}
Suppose that $n > p$ and the function $F$ is non-negative and non-decreasing with respect to the last argument on the set ${\mathbb R}^n \times [0, \infty)$ and positive on ${\mathbb R}^n \times (0, \infty)$.
Also let for any real number $\zeta_* > 0$ there be a real number $r_* > 0$, a locally bounded measurable function
$q : {\mathbb R}^n \setminus B_{r_*} \to [0, \infty)$,
and a non-decreasing function 
$g : [\zeta_*, \infty) \to (0, \infty)$
satisfying condition~$(c')$. If~\eqref{T2.4.1} is valid and, moreover,
\begin{equation}
	\int_{
		{\mathbb R}^n \setminus B_{r_*}
	}
	q (x)
	\,
	dx
	=
	\infty,
	\label{T4.3.1}
\end{equation}
then any non-negative solution of~\eqref{4.7}, \eqref{4.8} is identically zero.
\end{Theorem}

\begin{proof}
By contradiction, let there be a non-negative solution of~\eqref{4.7}, \eqref{4.8} that is not identically zero. According to Lemma~\ref{3.2}, this solution is positive almost everywhere in ${\mathbb R}^n$.
We put
$$
	E (r)
	=
	\int_{
		B_r
	}
	F (x, u)
	\,
	dx,
	\quad
	r > 0.
$$
By~\eqref{4.8}, there is a real number $r_0 > 0$ such that
$$
	\operatorname*{ess\,inf}\limits_{
		B_{r_0}
	}
	u
	<
	\varepsilon.
$$
We also denote $r_i = 2^i$, $i = 1,2,\ldots$.
Lemmas~\ref{L3.2} and~\ref{L3.3} lead to the estimate
$$
	\frac{
		1
	}{
		\mes B_{r_i}
	}
	\int_{
		B_{r_i}
	}
	F (x, u)
	\,
	dx
	\le
	C
	r_i^{- p}
	\left(
		\operatorname*{ess\,inf}\limits_{
			B_{r_i}
		}
		u
	\right)^{p - 1},
	\quad
	i = 1,2,\ldots,
$$
whence it follows that
$$
	\operatorname*{ess\,inf}\limits_{
		B_{r_i}
	}
	u
	\ge
	\sigma
	\left(
		\frac{
			E (r_i)
		}{
			r_i^{n - p} 
		}
	\right)^{1 / (p - 1)},
	\quad
	i = 1,2,\ldots.
$$
This in turn yields
$$
	F (x, u (x))
	\ge
	F
	\left(
		x,
		\operatorname*{ess\,inf}\limits_{
			B_{r_i}
		}
		u
	\right)
	\ge
	F
	\left(
		x,
		\sigma
		\left(
			\frac{
				E (r_i)
			}{
				r_i^{n - p} 
			}
		\right)^{1 / (p - 1)}
	\right)
$$
for almost all $x \in B_{r_i}$, $i = 1,2,\ldots$.
Hence, taking into account condition~$(c')$, we have
$$
	F (x, u (x))
	\ge
	q (x) 
	g (\varkappa E (r_i))
$$
for almost all $x \in B_{r_i} \setminus B_{r_{i-1}}$, $i = i_0 + 1, i_0 + 2,\ldots$,
where $i_0 \ge 0$ is some integer.
Integrating the last inequalities over $B_{r_i} \setminus B_{r_{i-1}}$, one can conclude that
$$
	E (r_i) - E (r_{i-1})
	\ge
	g (\varkappa E (r_i))
	\int_{
		B_{r_i} \setminus B_{r_{i-1}}
	}
	q (x)
	\,
	dx
$$
or, in other words,
$$
	\frac{
		E (r_i) - E (r_{i-1})
	}{
		g (\varkappa E (r_i))
	}
	\ge
	\int_{
		B_{r_i} \setminus B_{r_{i-1}}
	}
	q (x)
	\,
	dx
	\quad
	i = i_0 + 1, i_0 + 2, \ldots.
$$
Combining this with~\eqref{PT4.2.1}, we obtain
$$
	\int_{
		E (r_{i-1})
	}^{
		E (r_i)
	}
	\frac{
		d\zeta
	}{
		g (\varkappa \zeta))
	}
	\ge
	\int_{
		B_{r_i} \setminus B_{r_{i-1}}
	}
	q (x)
	\,
	dx
	\quad
	i = i_0 + 1, i_0 + 2, \ldots.
$$
Thus, summing the last expression over all $i > i_0 $, we arrive at the inequality
$$
	\int_{
		E (r_{i_0})
	}^\infty
	\frac{
		d\zeta
	}{
		g (\varkappa \zeta))
	}
	\ge
	C
	\int_{
		{\mathbb R}^n \setminus B_{r_{i_0}}
	}
	q (x)
	\,
	dx
$$
which contradicts~\eqref{T2.4.1} and~\eqref{T4.3.1}.
\end{proof}

\paragraph{4.4}
Consider the systems
\begin{equation}
	\left\{
		\begin{aligned}
			&
			- \operatorname{div} A_1 (x, \nabla u_1)
			\ge
			F_1 (x, u_1, u_2)
			&
			\mbox{in } {\mathbb R}^n,
			\\
			&
			- \operatorname{div} A_2 (x, \nabla u_2)
			\ge
			F_2 (x, u_1, u_2)
			&
			\mbox{in } {\mathbb R}^n,
		\end{aligned}
	\right.
	\label{4.4.1}
\end{equation}
where $A_1$ and $A_2$ are Caratheodory functions satisfying the uniform ellipticity condition~\eqref{ec} and, moreover, the functions $F_1$ and $F_2$ are non-negative and non-decreasing with respect to the last two arguments on the set ${\mathbb R}^n \times [0, \varepsilon] \times [0, \varepsilon]$ and positive on ${\mathbb R}^n \times (0, \varepsilon) \times (0, \varepsilon)$ for some real number $0 < \varepsilon < 1$.

We shall assume that for any real number $\zeta_* > 0$ there exist a real number $r_* > 0$ and functions
$q : [r_*, \infty) \to [0, \infty)$
and
$g : [\zeta_*, \infty) \to (0, \infty)$
satisfying the following condition:
\begin{enumerate}
\item[$(d)$]
for all $r \in [r_*, \infty)$ and $\zeta_1, \zeta_2 \in [\zeta_*, \infty)$ such that
$$
	\left(
		\frac{
			\zeta_1
		}{
			r^{n - p_1}
		}
	\right)^{1 / (p_1 - 1)}
	\le
	\varepsilon
	\quad
	\mbox{and}
	\quad
	\left(
		\frac{
			\zeta_2
		}{
			r^{n - p_2}
		}
	\right)^{1 / (p_2 - 1)}
	\le
	\varepsilon
$$
we have
$$
	f 
	\left( 
		r,
		\left(
			\frac{
				\zeta_1
			}{
				r^{n - p_1}
			}
		\right)^{1 / (p_1 - 1)},
		\left(
			\frac{
				\zeta_2
			}{
				r^{n - p_2}
			}
		\right)^{1 / (p_2 - 1)}
	\right)
	\ge
	q (r) 
	g (\zeta_1 + \zeta_2),
$$
\end{enumerate}
where
$$
	f (r, \zeta_1, \zeta_2)
	=
	\operatorname*{ess\,inf}\limits_{
		x \in B_{\theta r} \setminus B_{r / \theta}
	}
	(F_1 (x, \zeta_1, \zeta_2) + F_2 (x, \zeta_1, \zeta_2)),
	\quad
	r, \zeta_1, \zeta_2 > 0,
$$
with some given real number $\theta > 1$.

\begin{Theorem}\label{T4.4}
Suppose that $n > p$. 
Also let for any real number $\zeta_* > 0$ there exist a real number $r_* > 0$, a locally bounded measurable function
$q : [r_*, \infty) \to [0, \infty)$,
and a non-decreasing function 
$g : [\zeta_*, \infty) \to (0, \infty)$
satisfying conditions~$(d)$, \eqref{T2.4.1}, and~\eqref{T2.4.2}.
Then any non-negative solution of~\eqref{4.4.1}, \eqref{1.2} is identically zero.
\end{Theorem}

\begin{proof}
Assume by contradiction that $(u_1, u_2)$ is a non-negative solution of~\eqref{4.4.1}, \eqref{1.2} not equal to zero.
In view of Lemma~\ref{L3.2}, the functions $u_1$ and $u_2$ are positive almost everywhere in ${\mathbb R}^n$.  
As in the proof of Theorem~\ref{T2.1}, we denote
$\Omega_\varepsilon = \Omega_{1, \varepsilon} \cap \Omega_{2, \varepsilon}$,
where
$
	\Omega_{k, \varepsilon}
	=
	\{
		x \in {\mathbb R}^n
		:
		u_k (x) < \varepsilon
	\}
$,
$k = 1,2$.
By Lemma~\ref{L3.4}, there is a real number $r_0 > 0$ such that~\eqref{PT2.1.1} holds for all $r \ge r_0$.
We put $r_i = \theta^i r_0$, $i = 1,2,\ldots$.
Also let 
$$
	E_1 (r)
	=
	\int_{
		\Omega_\varepsilon \cap B_r
	}
	F_1 (x, u_1, u_2)
	\,
	dx,
	\quad
	r > 0,
$$
and
$$
	E_2 (r)
	=
	\int_{
		\Omega_\varepsilon \cap B_r
	}
	F_2 (x, u_1, u_2)
	\,
	dx,
	\quad
	r > 0.
$$

Corollary~\ref{C3.2} allows us to assert that
$$
	\frac{
		1
	}{
		\mes B_{r_i}
	}
	\int_{
		\Omega_\varepsilon \cap B_{r_i}
	}
	F_1 (x, u_1, u_2)
	\,
	dx
	\le
	C
	r_i^{- p_1}
	\left(
		\operatorname*{ess\,inf}\limits_{
			B_{r_i}
		}
		u_{1, \varepsilon}
	\right)^{p_1 - 1},
	\quad
	i = 1,2,\ldots,
$$
and
$$
	\frac{
		1
	}{
		\mes B_{r_i}
	}
	\int_{
		\Omega_\varepsilon \cap B_{r_i}
	}
	F_2 (x, u_1, u_2)
	\,
	dx
	\le
	C
	r_i^{- p_2}
	\left(
		\operatorname*{ess\,inf}\limits_{
			B_{r_i}
		}
		u_{2, \varepsilon}
	\right)^{p_2 - 1},
	\quad
	i = 1,2,\ldots,
$$
where the functions $u_{k, \varepsilon}$, $k = 1,2$, are defined by~\eqref{PT2.1.8}.
This immediately yields
\begin{equation}
	\operatorname*{ess\,inf}\limits_{
		\Omega_{1, \varepsilon} \cap B_{r_i}
	}
	u_1
	=
	\operatorname*{ess\,inf}\limits_{
		B_{r_i}
	}
	u_{1, \varepsilon}
	\ge
	\sigma
	\left(
		\frac{
			E_1 (r_i)
		}{
			r_i^{n - p_1} 
		}
	\right)^{1 / (p_1 - 1)},
	\quad
	i = 1,2,\ldots,
	\label{PT4.4.3}
\end{equation}
and
\begin{equation}
	\operatorname*{ess\,inf}\limits_{
		\Omega_{2, \varepsilon} \cap B_{r_i}
	}
	u_2
	=
	\operatorname*{ess\,inf}\limits_{
		B_{r_i}
	}
	u_{2, \varepsilon}
	\ge
	\sigma
	\left(
		\frac{
			E_2 (r_i)
		}{
			r_i^{n - p_2} 
		}
	\right)^{1 / (p_2 - 1)},
	\quad
	i = 1,2,\ldots.
	\label{PT4.4.4}
\end{equation}
Hence, taking into account the fact that $F_1$ and $F_2$ are non-decreasing functions with respect to the last two arguments on the set ${\mathbb R}^n \times [0, \varepsilon] \times [0, \varepsilon]$, we obtain
\begin{align}
	&
	F_k (x, u_1 (x), u_2 (x))
	\ge
	F_k
	\left(
		x,
		\operatorname*{ess\,inf}\limits_{
			\Omega_{1, \varepsilon} \cap B_{r_i}
		}
		u_1,
		\operatorname*{ess\,inf}\limits_{
			\Omega_{2, \varepsilon} \cap B_{r_i}
		}
		u_2
	\right)
	\nonumber
	\\
	&
	\qquad
	{}
	\ge
	F_k
	\left(
		x,
		\sigma
		\left(
			\frac{
				E_1 (r_i)
			}{
				r_i^{n - p_1} 
			}
		\right)^{1 / (p_1 - 1)},
		\sigma
		\left(
			\frac{
				E_2 (r_i)
			}{
				r_i^{n - p_2} 
			}
		\right)^{1 / (p_2 - 1)}
	\right),
	\quad
	k = 1,2,
	\label{PT4.4.2}
\end{align}
for almost all $x \in \Omega_\varepsilon \cap B_{r_i}$, $i = 1,2,\ldots$.
This implies that
\begin{align*}
	&
	F_1 (x, u_1 (x), u_2 (x))
	+
	F_2 (x, u_1 (x), u_2 (x))
	\\
	&
	\qquad
	{}
	\ge
	f
	\left(
		r,
		\left(
			\frac{
				\varkappa
				E_1 (r_i)
			}{
				r^{n - p_1} 
			}
		\right)^{1 / (p_1 - 1)},
		\left(
			\frac{
				\varkappa
				E_2 (r_i)
			}{
				r^{n - p_2} 
			}
		\right)^{1 / (p_2 - 1)}
	\right)
\end{align*}
for almost all $x \in \Omega_\varepsilon \cap B_{r_i} \setminus B_{r_{i-1}}$ and for all $r \in (r_{i-1}, r_i)$, $i = 1, 2,\ldots$.
Integrating the last expression, we have
\begin{align*}
	&
	\int_{
		\Omega_\varepsilon \cap B_{r_i} \setminus B_{r_{i-1}}
	}
	(
		F_1 (x, u_1, u_2)
		+
		F_2 (x, u_1, u_2)
	)
	\,
	dx
	\\
	&
	\qquad
	{}
	\ge
	C
	r_i^n
	f
	\left(
		r,
		\left(
			\frac{
				\varkappa
				E_1 (r_i)
			}{
				r^{n - p_1} 
			}
		\right)^{1 / (p_1 - 1)},
		\left(
			\frac{
				\varkappa
				E_2 (r_i)
			}{
				r^{n - p_2} 
			}
		\right)^{1 / (p_2 - 1)}
	\right)
\end{align*}
for all $r \in (r_{i-1}, r_i)$, $i = 1,2,\ldots$. Taking further condition~$(d)$ into account, we arrive at estimate~\eqref{PT2.2.3} with
\begin{equation}
	E (r) = E_1 (r) + E_2 (r),
	\quad
	r > 0.
	\label{PT4.4.1}
\end{equation}
Thus, to complete the proof, it remains to repeat the argument given in the proof of Theorem~\ref{T4.2}.
\end{proof}

Theorem~\ref{T4.4} can be strengthen if the $F_1$ and $F_2$ are non-decreasing functions with respect to the last two arguments on the whole set ${\mathbb R}^n \times [0, \infty) \times [0, \infty)$.

In this case, it should be assumed that for any real number $\zeta_* > 0$ there exist a real number $r_* > 0$ and functions
$q : {\mathbb R}^n \setminus B_{r_*} \to [0, \infty)$
and
$g : [\zeta_*, \infty) \to (0, \infty)$
satisfying the following condition:
\begin{enumerate}
\item[$(d')$]
for almost all $x \in {\mathbb R}^n \setminus B_{r_*}$ and for all $\zeta_1, \zeta_2 \in [\zeta_*, \infty)$ such that
$$
	\left(
		\frac{
			\zeta_1
		}{
			|x|^{n - p_1}
		}
	\right)^{1 / (p_1 - 1)}
	\le
	\varepsilon
	\quad
	\mbox{and}
	\quad
	\left(
		\frac{
			\zeta_2
		}{
			|x|^{n - p_2}
		}
	\right)^{1 / (p_2 - 1)}
	\le
	\varepsilon
$$
we have
\begin{align*}
	&
	F_1 
	\left( 
		x,
		\left(
			\frac{
				\zeta_1
			}{
				|x|^{n - p_1}
			}
		\right)^{1 / (p_1 - 1)},
		\left(
			\frac{
				\zeta_2
			}{
				|x|^{n - p_2}
			}
		\right)^{1 / (p_2 - 1)}
	\right)
	\\
	&
	\qquad
	{}
	+
	F_2
	\left( 
		x,
		\left(
			\frac{
				\zeta_1
			}{
				|x|^{n - p_1}
			}
		\right)^{1 / (p_1 - 1)},
		\left(
			\frac{
				\zeta_2
			}{
				|x|^{n - p_2}
			}
		\right)^{1 / (p_2 - 1)}
	\right)
	\ge
	q (x) 
	g (\zeta_1 + \zeta_2).
\end{align*}
\end{enumerate}
In so doing, we do not make any assumptions regarding the real number $\varepsilon \in (0, 1)$, except that it does not depend on $\zeta_*$.

\begin{Theorem}\label{T4.5}
Suppose that $n > p$ and the functions $F_1$ and $F_2$ are non-negative and non-decreasing with respect to the last two arguments on the set ${\mathbb R}^n \times [0, \infty) \times [0, \infty)$ and positive on ${\mathbb R}^n \times (0, \infty) \times (0, \infty)$.
Also let for any real number $\zeta_* > 0$ there be a real number $r_* > 0$, a locally bounded measurable function
$q : {\mathbb R}^n \setminus B_{r_*} \to [0, \infty)$,
and a non-decreasing function 
$g : [\zeta_*, \infty) \to (0, \infty)$ satisfying conditions~$(d')$, \eqref{T2.4.1}, and~\eqref{T4.3.1}. Then any non-negative solution of~\eqref{4.4.1}, \eqref{1.2} is identically zero.
\end{Theorem}

\begin{proof}
By contradiction, let $(u_1, u_2)$ be a non-negative solution of~\eqref{4.4.1}, \eqref{1.2} that is not equal to zero.
According to Lemma~\ref{L3.2}, both the functions $u_1$ and $u_2$ are positive almost everywhere in ${\mathbb R}^n$. 
In vie of~\eqref{1.2}, there exists a real number $r_0 > 0$ such that~\eqref{PT4.1.1} is valid for all $r \ge r_0$.  
Also let $r_i = \theta^i r_0$, $i = 1,2,\ldots$.

We denote
$$
	E_1 (r)
	=
	\int_{
		B_r
	}
	F_1 (x, u_1, u_2)
	\,
	dx,
	\quad
	r > 0,
$$
and
$$
	E_2 (r)
	=
	\int_{
		B_r
	}
	F_2 (x, u_1, u_2)
	\,
	dx,
	\quad
	r > 0.
$$

Corollary~\ref{C3.2} implies the estimates
$$
	\frac{
		1
	}{
		\mes B_{r_i}
	}
	\int_{
		B_{r_i}
	}
	F_1 (x, u_1, u_2)
	\,
	dx
	\le
	C
	r_i^{- p_1}
	\left(
		\operatorname*{ess\,inf}\limits_{
			B_{r_i}
		}
		u_1
	\right)^{p_1 - 1},
	\quad
	i = 1,2,\ldots,
$$
and
$$
	\frac{
		1
	}{
		\mes B_{r_i}
	}
	\int_{
		B_{r_i}
	}
	F_2 (x, u_1, u_2)
	\,
	dx
	\le
	C
	r_i^{- p_2}
	\left(
		\operatorname*{ess\,inf}\limits_{
			B_{r_i}
		}
		u_2
	\right)^{p_2 - 1},
	\quad
	i = 1,2,\ldots,
$$
whence it follows that
$$
	\operatorname*{ess\,inf}\limits_{
		B_{r_i}
	}
	u_1
	\ge
	\sigma
	\left(
		\frac{
			E_1 (r_i)
		}{
			r_i^{n - p_1} 
		}
	\right)^{1 / (p_1 - 1)},
	\quad
	i = 1,2,\ldots,
$$
and
$$
	\operatorname*{ess\,inf}\limits_{
		B_{r_i}
	}
	u_2
	\ge
	\sigma
	\left(
		\frac{
			E_2 (r_i)
		}{
			r_i^{n - p_2} 
		}
	\right)^{1 / (p_2 - 1)},
	\quad
	i = 1,2,\ldots.
$$
Taking into account the fact that $F_1$ and $F_2$ are non-decreasing functions with respect to the last two arguments on the set ${\mathbb R}^n \times [0, \infty) \times [0, \infty)$, we have
\begin{align*}
	F_k (x, u_1 (x), u_2 (x))
	&
	{}
	\ge
	F_k
	\left(
		x,
		\operatorname*{ess\,inf}\limits_{
			B_{r_i}
		}
		u_1,
		\operatorname*{ess\,inf}\limits_{
			B_{r_i}
		}
		u_2
	\right)
	\\
	&
	{}
	\ge
	F_k
	\left(
		x,
		\sigma
		\left(
			\frac{
				E_1 (r_i)
			}{
				r_i^{n - p_1} 
			}
		\right)^{1 / (p_1 - 1)},
		\sigma
		\left(
			\frac{
				E_2 (r_i)
			}{
				r_i^{n - p_2} 
			}
		\right)^{1 / (p_2 - 1)}
	\right),
	\quad
	k = 1,2,
\end{align*}
for almost all $x \in B_{r_i} \setminus B_{r_{i-1}}$, $i = 1,2,\ldots$.
By condition~$(d')$, this implies that
$$
	F_1 (x, u_1 (x), u_2 (x))
	+
	F_2 (x, u_1 (x), u_2 (x))
	\ge
	q (x)
	g (\varkappa (E_1 (r_i) + E_2 (r_i)))
$$
for almost all $x \in B_{r_i} \setminus B_{r_{i-1}}$, $i = 1,2,\ldots$.
Integrating the last inequality over $B_{r_i} \setminus B_{r_{i-1}}$, we obtain~\eqref{PT2.2.3} with $E$ defined by~\eqref{PT4.4.1}.
Thus, it remains to repeat the argument given in the proof of Theorem~\ref{T4.2}.
\end{proof}

\paragraph{4.5} 
In~\eqref{4.4.1}, let the functions $F_1$ and $F_2$ do not depend on $x$. In other words, we deal with the system
\begin{equation}
	\left\{
		\begin{aligned}
			&
			- \operatorname{div} A_1 (x, \nabla u_1)
			\ge
			F_1 (u_1, u_2)
			&
			\mbox{in } {\mathbb R}^n,
			\\
			&
			- \operatorname{div} A_2 (x, \nabla u_2)
			\ge
			F_2 (u_1, u_2)
			&
			\mbox{in } {\mathbb R}^n,
		\end{aligned}
	\right.
	\label{4.5.1}
\end{equation}
where $A_1$ and $A_2$ are Caratheodory functions satisfying the uniform ellipticity condition~\eqref{ec}. In so doing, the functions $F_1$ and $F_2$ are non-negative and non-decreasing with respect to every of their arguments on the set $[0, \varepsilon] \times [0, \varepsilon]$ and positive on $(0, \varepsilon) \times (0, \varepsilon)$, where $0 < \varepsilon < 1$ is some real number.

We shall assume that for any real number $\zeta_* > 0$ there exist real numbers $r_* > 0$ and $p > 1$ and a function
$f : (0, \infty) \to (0, \infty)$
satisfying the following condition:
\begin{enumerate}
\item[$(e)$]
for all $r \in [r_*, \infty)$ and $\zeta_1, \zeta_2 \in [\zeta_*, \infty)$ such that
$$
	\left(
		\frac{
			\zeta_1
		}{
			r^{n - p_1}
		}
	\right)^{1 / (p_1 - 1)}
	\le
	\varepsilon
	\quad
	\mbox{and}
	\quad
	\left(
		\frac{
			\zeta_2
		}{
			|x|^{n - p_2}
		}
	\right)^{1 / (p_2 - 1)}
	\le
	\varepsilon
$$
we have
\begin{align*}
	&
	F_1 
	\left( 
		\left(
			\frac{
				\zeta_1
			}{
				r^{n - p_1}
			}
		\right)^{1 / (p_1 - 1)},
		\left(
			\frac{
				\zeta_2
			}{
				r^{n - p_2}
			}
		\right)^{1 / (p_2 - 1)}
	\right)
	\\
	&
	\qquad
	{}
	+
	F_2
	\left( 
		\left(
			\frac{
				\zeta_1
			}{
				r^{n - p_1}
			}
		\right)^{1 / (p_1 - 1)},
		\left(
			\frac{
				\zeta_2
			}{
				r^{n - p_2}
			}
		\right)^{1 / (p_2 - 1)}
	\right)
	\ge
	f
	\left(
		\left(
			\frac{
				\zeta_1 + \zeta_2
			}{
				r^{n - p}
			}
		\right)^{1 / (p - 1)}
	\right).
\end{align*}
\end{enumerate}

\begin{Theorem}\label{T4.6}
Suppose that $n > \operatorname{max} \{ p_1, p_2 \}$. 
Also let for any real number $\zeta_* > 0$ there exist real numbers $r_* > 0$ 
and $p > 1$ and a non-decreasing function
$f : (0, \infty) \to (0, \infty)$ 
satisfying condition~$(e)$ and, moreover,
$$
	\int_0^\varepsilon
	\frac{
		f (\zeta)
		\,
		d\zeta
	}{
		\zeta^{1 + n (p - 1) / (n - p)}
	}
	=
	\infty.
$$
Then any non-negative solution of~\eqref{4.5.1}, \eqref{1.2} is identically zero.
\end{Theorem}

\begin{proof}
Let us follow the same argument and notation as in the proof of Theorem~\ref{T4.4}. The only difference is that $F_1$ and $F_2$ are now independent of the spatial variable.
Taking into account~\eqref{PT4.4.2}, we have
\begin{align*}
	&
	F_1 (u_1 (x), u_2 (x))
	+
	F_1 (u_1 (x), u_2 (x))
	\\
	&
	\quad
	{}
	\ge
	F_1
	\left(
		\sigma
		\left(
			\frac{
				E_1 (r_i)
			}{
				r_i^{n - p_1} 
			}
		\right)^{1 / (p_1 - 1)},
		\sigma
		\left(
			\frac{
				E_2 (r_i)
			}{
				r_i^{n - p_2} 
			}
		\right)^{1 / (p_2 - 1)}
	\right)
	\\
	&
	\qquad
	{}
	+
	F_2
	\left(
		\sigma
		\left(
			\frac{
				E_1 (r_i)
			}{
				r_i^{n - p_1} 
			}
		\right)^{1 / (p_1 - 1)},
		\sigma
		\left(
			\frac{
				E_2 (r_i)
			}{
				r_i^{n - p_2} 
			}
		\right)^{1 / (p_2 - 1)}
	\right)
\end{align*}
for almost all $x \in \Omega_\varepsilon \cap B_{r_i}$, $i = 1,2,\ldots$.
In view of condition~$(e)$, this implies that
$$
	F_1 (u_1 (x), u_2 (x))
	+
	F_1 (u_1 (x), u_2 (x))
	\ge
	f
	\left(
		\varkappa
		\left(
			\frac{
				E_1 (r_i) + E_2 (r_i)
			}{
				r_i^{n - p}
			}
		\right)^{1 / (p - 1)}
	\right)
$$
for almost all $x \in \Omega_\varepsilon \cap B_{r_i}$, $i = 1,2,\ldots$.
Integrating the last relation over $\Omega_\varepsilon \cap B_{r_i} \setminus B_{r_{i-1}}$, we arrive in accordance with~\eqref{PT2.1.1} at the following estimate:
$$
	E (r_i) - E (r_{i-1})
	\ge
	C
	r_i^n
	f
	\left(
		\varkappa
		\left(
			\frac{
				E (r_i)
			}{
				r_i^{n - p}
			}
		\right)^{1 / (p - 1)}
	\right),
	\quad
	i = 1,2,\ldots,
$$
where the function $E$ is defined by~\eqref{PT4.4.1}.
In particular, one can assert that $n > p$; otherwise we obtain
$$
	\operatorname*{lim\,inf}\limits_{i \to \infty}
	\frac{
		E (r_i)
	}{
		r_i^n
	}
	\ge
	C
	\operatorname*{lim\,inf}\limits_{i \to \infty}
	f
	\left(
		\varkappa
		\left(
			\frac{
				E (r_i)
			}{
				r_i^{n - p}
			}
		\right)^{1 / (p - 1)}
	\right)
	>
	0,
$$
which contradicts~\eqref{PT4.4.3} and~\eqref{PT4.4.4}.
Thus, to complete the proof, it suffices to repeat the argument given in the proof of~\cite[Theorem~2.1]{arXiv}.
\end{proof}

\paragraph{4.6} 
Let us note in conclusion that the method outlined in our paper is obviously suitable for systems containing more then two inequalities. It is also not difficult to transfer all the above results to Carnot groups.

\end{document}